\documentclass[mnsc,nonblindrev]{informs5}

\usepackage{amssymb,bm}
\usepackage{endnotes}

\usepackage{subcaption}

\usepackage[normalem]{ulem}

\usepackage{graphicx}
\usepackage{dsfont}
\usepackage{verbatim}
\usepackage{mathrsfs}
\usepackage{color}
\usepackage{epsfig}

\usepackage{url}
\usepackage{natbib}
\bibpunct[, ]{(}{)}{,}{a}{}{,}%
\def\bibfont{\small}%
\def\bibsep{\smallskipamount}%
\def\bibhang{24pt}%
\def\bibfont{\small}%
\def\bibsep{\smallskipamount}%
\def\bibhang{24pt}%

\usepackage[colorlinks,linkcolor=blue,citecolor=blue,anchorcolor=blue]{hyperref}

\newcommand{\Px}{ \mathbb{P} }

\newcommand{\Ex}{ \mathbb{E} }

\def\esssup_#1{\underset{#1}{\mathrm{ess\,sup\, }}}
\def\essinf_#1{\underset{#1}{\mathrm{ess\,inf\, }}}
\def\argmax_#1{\underset{#1}{\mathrm{arg\,max\, }}}
\def\argmin_#1{\underset{#1}{\mathrm{arg\,min\, }}}

\newcommand{\Fx}{\mathbb{F} }

\newcommand{\R}{\mathds{R}}

\newtheorem{theorem}{Theorem}[section]

\numberwithin{equation}{section}

\newtheorem{remark}[theorem]{Remark}
\newtheorem{lemma}[theorem]{Lemma}
\newtheorem{corollary}[theorem]{Corollary}

\allowdisplaybreaks[4]

\definecolor{Red}{rgb}{1.00, 0.00, 0.00}
\newcommand{\Red}{\color{Red}}
\definecolor{DRed}{rgb}{0.5, 0.00, 0.00}

\definecolor{Blue}{rgb}{0.00, 0.00, 1.00}

\definecolor{Green}{rgb}{0.0, 0.4, 0.0}

\newcolumntype{I}{!{\vrule width 1.2pt}}
\newlength\savedwidth

\newlength\savewidth

\usepackage{amssymb}

\begin{document}
\RUNAUTHOR{Bo, Huang and Zhang}
	\RUNTITLE{Optimal Consumption and Retirement Time under Shortfall Risk Measure}
	
	\TITLE{Optimal Consumption and Retirement Time under Shortfall Risk Measure}
	\ARTICLEAUTHORS{%
 \AUTHOR{Lijun Bo}
	\AFF{School of Mathematical Sciences,  University of Science and Technology of China, Hefei,  China\\ \EMAIL{lijunbo@ustc.edu.cn}}
\AUTHOR{Yijie Huang}
	\AFF{Department of Applied Mathematics, The Hong Kong Polytechnic University, Kowloon, Hong Kong, China\\ \EMAIL{yijie.huang@polyu.edu.hk}}
	\AUTHOR{Tingting Zhang}
	\AFF{Center for Financial Engineering,  Soochow University,  Suzhou,  China\\ \EMAIL{ttzhang1118@suda.edu.cn}  
}
	} 

\ABSTRACT{
This paper studies the optimal portfolio, consumption, and endogenous early retirement problem within a benchmark tracking framework by incorporating a new relative performance evaluation.  
In this framework, the investor maximizes expected lifetime consumption utility while managing the maximum wealth shortfall relative to a benchmark, with shortfall-management costs that may differ before and after retirement. Mathematically, the problem is a hybrid stochastic control problem involving both regular controls and an optimal stopping time, in which the running maximum process records the investor's largest benchmark shortfall. We introduce an auxiliary reflected state process and establish an equivalent hybrid stochastic control problem. By proving the convex duality theorem, we technically transform the original problem into a two-dimensional pure optimal stopping problem with state reflection. This enables us to characterize the geometric structure of the stopping set and derive the feedback-form optimal retirement boundary, as well as optimal portfolio and consumption policies. Analytical examples and numerical simulations reveal a two-stage structure with more conservative investment and more aggressive consumption after retirement. Driven by the retirement option, the expected largest shortfall risk follows a pronounced U-shaped pattern with respect to wealth. Shortfall management costs, labor income, and leisure preference significantly influence retirement timing, investment, and consumption. 
}
\maketitle

\section{Introduction}
Portfolio optimization lies at the heart of investment management, wherein investors dynamically adjust asset allocations, consumption plans, and long-term financial strategies to maximize welfare while monitoring performance relative to pre-specified benchmark indices. Since the seminal contributions of \cite{merton1969lifetime,merton1971optimum} establishing the continuous-time framework for lifetime portfolio and consumption decisions, the literature has extensively extended the baseline model to incorporate various real-world economic frictions, including stochastic labor income, transaction costs, and portfolio constraints.  A prominent strand of this research focuses on benchmark tracking and performance constraints, which require the portfolio wealth to outperform or remain above a stochastic reference process, a critical objective for individual investors who seek to control the risk of underperformance relative to their financial goals.  Unlike institutional investors who may access external financing, individual investors interpret benchmark tracking as a risk management objective. They aim to penalize the shortfall risk when wealth falls below the benchmark, representing their aversion to underperformance rather than actual capital infusion. Within the stochastic control framework, this shortfall risk is mathematically captured by a maximum process, which records the cumulative largest shortfall when wealth falls below the benchmark, and the associated cost parameter quantifies the investor's disutility from benchmark underperformance. Recent studies (\citealt{bo2025stochastic,bo2026optimal,bo2026extended}) have proposed relaxed tracking formulations that balance consumption utility maximization with shortfall risk minimization. However, existing literature has largely overlooked the systematic investigation of life-cycle decisions, particularly endogenous retirement timing under benchmark tracking constraints.

With the increasing complexity of financial markets and the accelerating global demographic aging, life-cycle financial planning has attracted substantial scholarly attention (\citealt{jin2006disutility,choi2008optimal,dybvig2010lifetime,farhi2007saving,yang2018optimal}). A key element therein is the early retirement option, an endogenous optimal stopping decision that reflects the investor's intricate trade-off between the utility increment from labor supply and the leisure value from financial freedom. Crucially, the marginal cost of shortfall exhibits significant heterogeneity across different stages of the investor's life-cycle. After retirement, the reduction in regular income and increased financial vulnerability amplify the adverse impact of benchmark underperformance, rendering the investor more sensitive to shortfall risk.
This life-cycle heterogeneity in shortfall aversion results in investors facing higher effective shortfall costs in the post-retirement phase than during their working years, establishing a close intrinsic connection between the optimal timing of retirement and minimizing extreme shortfall exposure. This motivates our analysis of the integrated decision problem encompassing optimal portfolio allocation, consumption choice, and retirement timing under benchmark-tracking-based shortfall risk measurement.

In this paper, we study a class of hybrid optimal stopping control problems within a benchmark tracking framework by incorporating a new relative performance evaluation. A key challenge in solving such hybrid control problems lies in the presence of a maximum process that captures the investor’s largest wealth shortfall relative to the benchmark, including non-linearity and dimensionality issues that make analytical solutions difficult to obtain.
To solve this problem, we introduce an auxiliary state process and establish an equivalent hybrid optimal stopping and control problem under state reflection. Furthermore, by proving the convex duality theorem for the auxiliary stochastic control problem, we technically transform the original hybrid control problem into a tractable two-dimensional pure optimal stopping problem with state reflection. 
This transformation overcomes the non-linearity challenge and allows us to characterize the geometric structure of the two-dimensional stopping set, which determines the optimal retirement timing. We study the corresponding variational inequality with Neumann boundary conditions for the dual problem, which is easier to handle compared to the nonlinear parabolic differential operator arising in the original problem. Based on the convex duality theorem, we obtain the feedback-form optimal retirement boundary for the original problem, as well as the optimal portfolio and consumption policies. We provide two analytical examples to analyze the properties of retirement decisions, optimal strategies, and relative wealth shortfall risk, and we implement numerical simulations to analyze and verify our theoretical findings, providing practical insights for investors and policymakers. This, in turn, underpins five key results. 

First, the investor’s optimal retirement timing is characterized by a wealth-dependent retirement threshold, which also varies with the benchmark state. Once wealth reaches this threshold, retirement becomes optimal, yielding higher expected utility.
Second, optimal portfolio and consumption policies exhibit a distinct two-stage structure around the retirement threshold. After retirement, the optimal portfolio becomes markedly more conservative than in the working stage. While employed, the investor benefits from steady labor income and faces lower shortfall costs, supporting higher consumption than in the immediate-retirement case.
Third, higher pre-retirement shortfall management costs reduce both risky investment and consumption, and bring forward the optimal retirement date. By contrast, higher post-retirement shortfall costs encourage greater pre-retirement consumption and delay retirement. A higher labor income increases pre-retirement risk tolerance but depresses consumption and raises the required retirement wealth threshold. A stronger preference for leisure lowers the retirement threshold and reduces pre-retirement consumption to accelerate wealth accumulation.
Fourth, the expected largest shortfall risk is U-shaped in wealth: it decreases before retirement and rises thereafter. Greater shortfall management costs increase the minimum shortfall risk and shift the U-shaped curve rightward, while higher labor income and stronger leisure preference reduce the minimum shortfall risk.
Fifth, a higher benchmark state leads to more conservative portfolios, lower consumption, and delayed retirement. It also raises and reshapes the shortfall risk profile, reflecting stronger precautionary behavior and a higher wealth requirement to meet stricter tracking targets.

The benchmark tracking problem has been extensively studied in the quantitative finance literature, with diverse formulations arising from different performance measures and constraint specifications. \cite{browne1999beating,browne1999reaching,browne2000risk} pioneers the study of active portfolio management with several objective functions, including maximizing the probability of beating the benchmark,  maximizing the probability of reaching a preset financial goal within a given time horizon, and their mixed formulations. An alternative approach, widely adopted in portfolio management practice, minimizes the tracking error variance or downside variance relative to the index value or return, leading to linear-quadratic stochastic control problems \citep{gaivoronski2005optimal,yao2006tracking,ni2022optimal}. 
\cite{strub2018optimal} introduce a tracking formulation that measures the similarity between normalized historical trajectories of the tracking portfolio and the index, incorporating rebalancing transaction costs.  \cite{bo2021optimal,bo2025stochastic} develop a relaxed tracking framework for monotone benchmark processes via fictitious capital injection, which imposes a dynamic floor constraint on wealth relative to the benchmark. \cite{bo2026extended} further extend this model to general It\^o's benchmarks and obtained semi-analytical optimal controls under CRRA utility, showing that capital injection expands the control set and leads to more aggressive strategies than the standard Merton solution. Additionally,  existing literature also investigates the related shortfall risk measures, but they differ essentially from the expected largest shortfall considered in this paper. For example,  \cite{pham2002minimizing} introduce the traditional expected shortfall for terminal random variables, which focuses on tail losses at a specific terminal time rather than the cumulative maximum shortfall over the entire time horizon. More recently, \cite{farkas2021intra} propose the intra-horizon expected shortfall, which captures losses occurring at any point during the trading period but does not integrate relative performance against a benchmark tied to retirement options. Notably, these studies are exclusively focused on infinite-horizon problems and do not incorporate life-cycle decisions like retirement timing, nor do they explore the differential shortfall risk management costs across various life-cycle stages.

The optimal retirement timing has emerged as an important research topic at the intersection of labor economics and financial mathematics, particularly against the backdrop of population aging and the increasing importance of individual retirement planning. The early contributions by \cite{jin2006disutility}, \cite{choi2008optimal}, and \cite{dybvig2010lifetime,dybvig2011verification} study voluntary retirement problems in infinite-horizon settings, characterizing the optimal retirement threshold as a critical wealth level above which the agent chooses to retire immediately. \cite{farhi2007saving} examine the interaction between retirement timing and savings decisions, providing empirical evidence that stock market performance significantly affects retirement patterns. For instance, the stock market boom between 1995 and 2000 precipitated a dramatic increase in voluntary early retirement. Notable advancements include the introduction of features like mandatory retirement dates and early retirement options \citep{yang2018optimal}, the consideration of consumption ratcheting \citep{jeon2020optimal}, the exploration of partial information \citep{chen2022optimal}, the examination of return ambiguity in risky asset prices \citep{park2023robust}, the analysis of non-Markovian environments \citep{yang2021optimal}, and the study of heterogeneous consumption patterns involving basic and luxury goods \citep{jang2024optimal}.
However, the above retirement literature has not considered the practical requirement of tracking stochastic benchmarks, nor has it incorporated measuring the largest shortfall risk. Distinct from those studies, our paper investigates the optimal retirement timing, investment, and consumption decisions under the shortfall risk measure, which mathematically involves a complex mixed optimal stopping and control problem with state reflection.

 In the literature, three common methods are used to analyze the properties of the value function and optimal strategy for optimal control problems: the martingale method with dual transformation (see, e.g., \citealt{cox1989optimal,karatzas1987optimal,karatzas1998methods}), the transformation of the associated Hamilton-Jacobi-Bellman (HJB) equation into a linear partial differential equation (PDE) via Legendre transformation (see, e.g., \citealt{karatzas1986explicit,fleming2006controlled}), and the stochastic maximum principle (see, e.g., \citealt{yong1999stochastic}). However, none of these methods can be directly applied to our problem. This is because the state equations differ substantially before and after retirement, and the state reflection measured by the wealth shortfall is also distinct, leading to different dual transformations and Legendre transformations in the two phases. Furthermore, the stopping time (i.e., retirement timing) is unknown and dynamically interacts with the optimal portfolio and consumption controls, as well as the minimization of the expected largest wealth shortfall risk. For optimal stopping problems alone, three commonly used methods exist: the martingale method (see, e.g., \citealt{karatzas1998methods,karatzas2000utility}), the probabilistic method (see, e.g., \citealt{peskir2006optimal,de2015nonconvex}), and the PDE method via its associated variational inequality (see, e.g., \citealt{bensoussan1974nonlinear,friedman1973stochastic}). However, for the mixed control problem investigated in our paper, it is difficult to characterize the reflected state process and the dual form of the objective function by directly adopting the martingale approach or probabilistic method. The corresponding HJB equation is a variational inequality governed by a fully nonlinear parabolic differential operator and subject to Neumann boundary conditions, which poses substantial challenges for analyzing the properties of the solution. This further highlights the necessity of our innovative methodology. Instead of directly tackling the original problem, this paper first introduces a heuristic process with state reflection and the corresponding value function. After systematically analyzing the structural properties and solvability of this heuristic process and its value function, we verify the dual relationship between them and the original problem. Subsequently, the optimal strategy and the solution structure of the original problem are derived from the dual value function. 

The rest of the paper is organized as follows. Section~\ref{sec:model} presents the problem formulation. Section~\ref{sec:duality} introduces the dual problem and associated structure properties. 
Section~\ref{sec:dualityTH} develops the convex duality theory for the original mixed optimal stopping-control problem and characterizes the corresponding optimal control strategy. Section~\ref{sec:exa}  provides two explicit examples and characterizes an investor's benefit and the expected largest shortfall risk from the retirement option. Section~\ref{sec:conc}  summarizes the main results and discusses future research
directions. The proofs of all results in the previous sections are collected in the E-Companion.

\section{Problem Formulation}\label{sec:model}
We consider a portfolio-consumption problem in which an investor tracks a stochastic benchmark while managing consumption and retirement decisions under downside shortfall risk. The problem leads to a hybrid framework of stochastic control and optimal stopping. More precisely, the market considered in the paper is established on a filtered probability space $(\Omega, \mathcal{F}, \Fx,\mathbb{P})$ with the filtration $\mathbb{F}=(\mathcal{F}_t)_{t\geq 0}$ satisfying the usual conditions. The market consists of $d$ risky assets and the price dynamics of asset $i$ is described as the Black-Scholes model, for $t\geq0$,
\begin{align}\label{stockSDE}
\frac{dS_t^i}{S_t^i}= \mu_idt+\sum_{j=1}^d\sigma_{ij}dW_t^{j},\quad i=1,\ldots,d,
\end{align}
where $\mu=(\mu_1,\ldots,\mu_d)^{\top}\in\R^d$ is the vector of return rates, $\sigma=(\sigma_{ij})_{i,j=1}^d$ denotes the volatility matrix which is invertible, and $W=(W^1,\ldots,W^d)^{\top}=(W_t^1,\ldots,W_t^d)_{t\geq 0}^{\top}$ is a $d$-dimensional $\mathbb{F}$-adapted Brownian motion (BM). Here, we assume that the riskless interest rate is zero, which amounts to the change of num\'{e}raire.
From this point onwards, all processes, including the wealth process and the benchmark process, are defined after the change of num\'{e}raire. 

Consider the performance of the investor's portfolio relative to an external benchmark process $Z = (Z_t)_{t\geq0}$ which is given by the following Black-Scholes model:
\begin{align}\label{eq:Zt}
d Z_t=\mu_Z Z_t d t+\sigma_Z Z_t  d W_t^\gamma,\quad Z_0=z\in\R_+:=[0,\infty),
\end{align}
where the BM $W^{\gamma}:=\gamma^{\top}W$ is a linear combination of the $d$-dimensional BM $W$ with weights $\gamma=(\gamma_1,\ldots,\gamma_d)^{\top}\in[-1,1]^d$ satisfying $|\gamma|=1$. We also assume that the weight $\gamma$ has no linear correlation with the shape ratio $\mu \sigma^{-1}$, which can guarantee the ellipticity of the operator of the variational inequality in Section~\ref{sec:duality}. Unless specified otherwise, $|\cdot|$ refers to the Euclidean norm of vectors. The benchmark process $Z=(Z_t)_{t\geq0}$ may represent a market index, a pension target, or a professionally managed reference portfolio used to evaluate the investor’s relative performance.

The investor dynamically chooses portfolio allocation, consumption, and retirement timing in order to balance utility from consumption against persistent underperformance relative to a benchmark. For $t\geq 0$, 
let $\theta_t^i$ denote the amount invested in risky asset $S^i$, and let $c_t$ denote the consumption rate.  
The investor receives stochastic labor income before retirement and has no income after retirement, where the labor income at time $t$ is a linear function of the index value he tracks at time $t$, i.e., the labor income at time $t$ is given by $r_t:=r Z_t+r_c$ with $r ,r_c> 0$. Then, the resulting self-financing wealth process of the investor under the portfolio vector $\theta=(\theta_t^1,\ldots,\theta_t^d)_{t\geq 0}^{\top}$ and the consumption strategy $c=(c_t)_{t\geq 0}$ is given by
\begin{align}\label{eq:wealth2}
d V^{\theta,c,\tau}_t =\theta_t^{\top}\mu dt- c_tdt+ r_t\mathds {1}_{\{\tau>t\}}dt+\theta_t^{\top}\sigma dW_t,\quad V^{\theta,c,\tau}_0={\rm v}\geq0,
\end{align}
where $\tau\in[0,+\infty]$ is the retirement time which is an $\Fx$-stopping time. We assume that
\begin{itemize}
\item[$\left(\boldsymbol{A}_{\boldsymbol{Z}}\right)$]
$0\leq r\leq \mu_X:=\mu_Z -\sigma_Z  \gamma^{\top} \sigma^{-1} \mu$.
\end{itemize}
Assumption $\left(\boldsymbol{A}_{\boldsymbol{Z}}\right)$ ensures that the market index’s Sharpe ratio is sufficiently high. Economically, this means that the investor in our model restricts attention to benchmark processes that exhibit strong market performance. In practice, investors naturally prefer benchmarks with high Sharpe ratios, making this assumption economically plausible. 

Given the benchmark process $Z=(Z_t)_{t\geq 0}$, denote by $A=(A_t)_{t\geq0}$ the largest shortfall when the wealth falls below the benchmark, that is
 \begin{align}\label{A-sing}
A^{(\theta,c,\tau)}_t =0\vee \sup_{s\leq t}\left(Z_{s}-V_{s}^{\theta,c,\tau}\right), \quad\forall t\geq0.
\end{align}
The process $A$, which is non-decreasing and non-negative, records the running maximal benchmark deviation of the investor’s portfolio, with higher values signaling greater exposure to persistent underperformance and associated shortfall risk under the given investment, consumption, and retirement strategy.
The associated shortfall risk is penalized through $\Ex[A_0+\int_0^{\cdot} e^{-\hat{\rho} t}dA_t]$, 
where $\hat{\rho}>0$ is the discount factor. 
By tracking the relative performance against the benchmark, the investor can assess the effectiveness of their investment strategies, strategically adjust asset allocations, consumption rates, and retirement time, thus optimizing the tradeoff between the expected utility of consumption and the expected largest shortfall. 
The investor considers maximizing his net utility of consumption, for any $(\mathrm{v},z)\in\R_+^2$,
\begin{align}\label{eq_prob_IBP0}
&{\rm w}(\mathrm{v},z):=\sup_{(\tau,\theta,c) \in \mathbb{T}\times\mathbb{U}^{\rm r}} \mathbb{E}\Bigg[\int_0^{\tau_d} e^{-\hat \rho s} (U(c_s)-\ell_s \mathds 1_{\{s\leq \tau\}} )ds-\alpha A_0\notag\\
&\qquad\qquad\qquad\qquad\qquad\quad-\left(\int_0^{\tau_d} e^{-\hat \rho s}(\alpha \mathds 1_{\{s\leq\tau\}}+\beta \mathds 1_{\{s>\tau\}})dA_s\right)\Bigg],
\end{align}
where we recall that $\hat{\rho}> 0$ is the discount rate, and  $\ell_t:=\ell Z_t+\ell_c>0$ with $\ell,\ell_c> 0$ representing the utility of leisure; while $\tau_d>0$ denotes the stochastic time of death and follows an exponential distribution with mortality rate $\delta>0$ which is independent of BM $W$. Here, $\mathbb{U}^{\rm r}$ is the admissible set of regular controls, which consists of all $\mathbb{F}$-adapted processes $(\theta,c)=(\theta_t, c_t)_{t \geq 0}$, taking values in $\R^d \times \R_+$, that satisfy the integrability condition $\Ex[\int_0^t (c_s + |\theta_s|^2) ds] < \infty$ for all $t \geq 0$. A retirement time $\tau$ is said to be admissible, i.e, $\tau\in\mathbb{T}$, if it is an $\mathbb{F}$-stopping time. Then, the value function \eqref{eq_prob_IBP0} can be rewritten as follows: 
\begin{align}\label{eq_prob_IBP}
{\rm w}(\mathrm{v},z)=\sup_{(\tau,\theta,c) \in \mathbb{T}\times\mathbb{U}^{\rm r}} \mathbb{E}\Bigg[&\int_0^\tau e^{- \rho s} (U(c_s)-\ell_s )ds-\alpha\left(A_0+\int_0^\tau e^{-  \rho s} dA_s\right)\nonumber\\
&+\int_\tau^{\infty} e^{-  \rho s} U(c_s)ds-\beta\int_\tau^{\infty} e^{- \rho s} dA_s\Bigg],
\end{align}
where $\rho:=\hat \rho+\delta$ is the updated discount factor. When the agent is engaged in work activities, a sacrifice of this leisure utility inevitably occurs. Regarding the financing costs, $\alpha>0$ signifies the utility per irreversible investment before retirement, while $\beta>0$ denotes the cost per unit of shortfall risk for the post-retirement. Note that $\alpha$ is typically not larger than $\beta$, i.e., $\alpha\leq \beta$. This is because, before retirement, agents often have stable income streams from employment, which provide lenders with a certain level of security and thus result in relatively lower borrowing costs. In contrast, after retirement, the absence of regular employment-based income increases the perceived risk for lenders, leading to a higher $\beta$. The utility function $U(\cdot):\R_+\to\R_+$ is the CRRA utility function given by
\begin{align}\label{eq:utility}
U(x) = 
\dfrac{1}{p}x^p, \quad -\infty< p < 1 \text{ and } p \neq 0, 
\end{align}
where the risk-aversion level of the investor is given by $1 - p \in (0,1)\cup (1, +\infty)$.  We have the following assumption on the discount factor $\rho$.
\begin{itemize}
\item
[$\left(\boldsymbol{A}_{\boldsymbol{\rho}}\right)$] The discount factor $\rho>0$ satisfies
$\rho >\max \left\{2\mu_Z+\sigma_Z^2, \xi p /(1-p)\right\}$ with $\xi:=\mu^{\top}(\sigma \sigma^{\top})^{-1} \mu / 2$.
\end{itemize}

To solve the hybrid optimal stopping problem \eqref{eq_prob_IBP} involving a running maximum process, we now introduce an auxiliary stochastic control problem which is equivalent to the original one. To do it, we replace the original state process $V^{\theta,c,\tau}=(V_t^{\theta,c,\tau})_{t\geq 0}$  given in \eqref{eq:wealth2} by a reflected controlled state process $X=(X_t)_{t\geq0}$. We then define the  difference process by 
\begin{align*}
D_t:=Z_t-V_t^{\theta,c,\tau}+\mathrm{v}-z,\quad \forall t\geq0.    
\end{align*}
Moreover, for any $x\geq0$, we introduce the running maximum process associated with $D=(D_t)_{t\geq 0}$ given by
\begin{align*}
L_t^X :=x\vee \sup_{s\leq t}D_s,\quad \forall t\geq0. \end{align*}
Thus, the new state process $X=(X_t)_{t\geq 0}$ is defined as the reflected process $X_t:=L_t^X-D_t$ for $t\geq0$, which has the following dynamics by using  \eqref{eq:Zt} and \eqref{eq:wealth2},
\begin{align}\label{state-X}
X_t &=\int_0^t\theta_s^{\top}\mu ds+\int_0^t (r Z_s+r_c) \mathds {1}_{\{s<\tau\}} ds+\int_0^t\theta_s^{\top}\sigma dW_s  -\int_0^t c_s ds-\int_0^t \mu_Z Z_sds\nonumber\\
&\quad-\int_0^t \sigma_Z Z_s{ \gamma^{\top}dW_s}+ L_t^X
\end{align}
with the initial value $X_0=L_0^X=x$. In particular, the running maximum process $L^X=(L_t^X)_{t\geq0}$, referred to as the local time of the state process $X=(X_t)_{t\geq0}$ at boundary zero, increases strictly iff $X_t=0$ (i.e., $L_t^X=D_t$). We obtain from \eqref{A-sing} that the hybrid stochastic control problem can be transformed into the following regular stochastic control problem, for $(x,z)\in\R_+^2$,
\begin{align}\label{eq:u}
&v(x,z):=\sup_{(\tau,\theta,c) \in \mathbb{T}\times\mathbb{U}^{\rm r}}J(x,z;\tau,\theta,c)\\
&:= \sup_{(\tau,\theta,c) \in \mathbb{T}\times\mathbb{U}^{\rm r}}\Ex\left[ \int_0^\tau e^{-\rho t} (U(c_t)-\ell Z_t-\ell_c )dt- \alpha \int_0^\tau e^{-\rho t}dL_t^X+\int_\tau^\infty e^{-\rho t} U(c_t)dt- \beta \int_\tau^{\infty} e^{-\rho t}dL_t^X\right],\nonumber
\end{align}
where the state process $(X,Z)=(X_t,Z_t)_{t\geq0}$ satisfies the dynamics~\eqref{state-X}~and~\eqref{eq:Zt}.

\begin{remark}\label{rem:Lipvxz}
In fact, the stochastic control problems \eqref{eq_prob_IBP} and \eqref{eq:u} are equivalent in the sense that ${\rm w}(\mathrm{v}, z)=-\alpha ( z-\mathrm{v})^++v((\mathrm{v}-z)^+, z)$.
\end{remark}

\section{The Dual Problem with State Reflection}\label{sec:duality}

This section constructs a convex duality framework to analyze the hybrid optimal stopping-control problem \eqref{eq:u} with state reflection. The core methodology hinges on introducing a rigorously designed dual process; this construction delivers an analytically more tractable dual stopping formulation and lays the groundwork for characterizing optimal consumption rules, portfolio allocation strategies, and optimal retirement timing.

\subsection{The Dual Stopping Problem}

We first introduce the so-called reflected dual process (RDP). More precisely, let the process $Y^{\alpha}=(Y^{\alpha}_t)_{t\geq0}$ which takes values on $(0,\alpha]$ satisfy the following reflected dynamics, for $t>0$, 
\begin{align}\label{eq:Y_t}
    dY^{\alpha}_t=\rho Y^{\alpha}_tdt-\mu^\top (\sigma\sigma^\top)^{-1}\sigma Y^{\alpha}_tdW_t-dL_t^{\alpha},\quad Y^{\alpha}_0\in(0,\alpha],
\end{align}
where $L^{\alpha}=(L_t^{\alpha})_{t \geq 0}$ is a continuous and non-decreasing process with zero initial value, which increases strictly on the time set $\{t\geq 0;~Y_t^{\alpha}=\alpha\}$ only. Moreover,  for a given stopping time $\tau\in\mathbb{T}$, we define the process $(Y_t^{\tau,\beta})_{t\geq\tau}$ as $Y^{\beta}_\tau=Y_{\tau}^{\alpha}$, and for $t>\tau$,
 \begin{align}\label{eq:Y_tbeta}
    dY^{\tau,\beta}_t=\rho Y^{\tau,\beta}_tdt-\mu^\top (\sigma\sigma^\top)^{-1}\sigma Y^{\tau,\beta}_tdW_t-dL_t^{\tau,\beta},
\end{align}
where $L^{{\tau,\beta}}=(L_t^{{\tau,\beta}})_{t>\tau}$ is a continuous and non-decreasing process, which increases strictly on the time set $\{t> \tau;~Y_t^{{\tau,\beta}}=\beta\}$ only.
Then, we have the following characterization of admissible portfolio-consumption and retirement time
from the RDP defined by \eqref{eq:Y_t}-\eqref{eq:Y_tbeta}.

\begin{lemma}\label{lem:XY} 
Let Assumption {\bf(A$_Z$)} and $\left(\boldsymbol{A}_{\boldsymbol{\rho}}\right)$  hold. For any strategy pair $(\tau,\theta,c) \in \mathbb{T}\times\mathbb{U}^{\rm r}$, consider the corresponding controlled state process $(X, Z)=\left(X_t, Z_t\right)_{t \geq 0}$ given by \eqref{state-X} and \eqref{eq:Zt} with initial value $\left(X_0, Z_0\right)=(x, z) \in \mathbb{R}_{+}^2$, and the RDP given by \eqref{eq:Y_t}-\eqref{eq:Y_tbeta} with initial value $Y_0^{\alpha}=y\in(0,\alpha]$. Then, it holds that
\begin{align}\label{eq:contrain}
&\mathbb{E}\bigg[\int_0^{\tau} e^{-\rho t}\left(c_t +(\mu_X-r) Z_t-r_c\right)Y_t^{\alpha} d t
+\int_0^{\tau} e^{-\rho t} X_t d L_t^{\alpha}-\int_0^{\tau} e^{-\rho t} Y_t^\alpha d L_t^X \notag\\
&\qquad+\int_{\tau}^\infty e^{-\rho t}\left(c_t +\mu_X Z_t \right)Y_t^{\tau,\beta} d t
+\int_\tau^{\infty} e^{-\rho t} X_t d L_t^{\tau,\beta}-\int_{ \tau} ^{\infty} e^{-\rho t} Y_t^{{\tau,\beta}} d L_t^X\bigg] \leq x y.
\end{align}
\end{lemma}

\begin{remark}
When the benchmark process $Z \equiv 0$ (constant zero benchmark) and the shortfall risk management cost  parameters $\alpha=\beta \to +\infty$, the local time processes vanish: $L^{\alpha}=L^{\tau,\beta}\equiv0$ (no reflection). Hence, the reflected dual process becomes the state price density process in the Black-Scholes model.
\end{remark}

Let $\widetilde{U}:(0,\beta]\to\R$ be the Legendre-Fenchel (LF) transform of the CRRA utility function $U(\cdot):\R_+\to\R_+$ given by \eqref{eq:utility}, i.e., $\tilde U(y):=\sup_{x> 0}\{U(x)-x y\}$ for $y\in(0,\beta]$. Then, we have
\begin{align}\label{eq:tildeU2}
    \tilde{U}(y)=
     \displaystyle \frac{1-p}{p}y^{\frac{p}{p-1}}.
\end{align}
Subsequently, we introduce the following (pure) optimal stopping problem, which is indeed the dual problem of the original hybrid stochastic control problem \eqref{eq:u}. For all $(y,z)\in (0,\alpha]\times\R_+$,
\begin{align}\label{eq:tildev}
& \tilde{v}(y, z):=\sup_{\tau\in\mathbb T}\tilde{J}(y,z;\tau)\\
&\quad=\sup_{\tau\in\mathbb T}\mathbb{E}\left[\int_0^\tau e^{-\rho t} \left(\tilde U(Y_t^\alpha)+(r-\mu_X) Z_t Y_t^\alpha-\ell Z_t+r_cY_t^\alpha-\ell_c\right)dt+e^{-\rho \tau}\tilde G(Y_\tau^\alpha,Z_\tau)
 \big|Y_0^\alpha=y,Z_0=z\right],\nonumber
\end{align}
where the post-retirement value function $\tilde{G}:(0, \beta] \times \mathbb{R}_+ \mapsto \mathbb{R}$ is defined by
\begin{align}\label{eq:tildeG}
 \tilde {G}(y,z):=\mathbb{E} \left[\int_0^{\infty} e^{-\rho t}\left(\tilde U(Y_t^{\beta})-\mu_X Z_tY_t^{\beta} \right)dt \mid Y_0^{\beta}=y,Z_0=z\right].
\end{align}
Here, the process $Y^{\beta}=(Y^{\beta}_t)_{t\geq0}$ taking values on $(0,\beta]$ satisfies RDP \eqref{eq:Y_t} with $\alpha$ replaced by $\beta$. 
By choosing immediate retirement in
\eqref{eq:tildev}, we obtain the lower bound $\tilde v(y,z)\geq \tilde G(y,z)$ for all $(y,z)\in(0,\alpha]\times\R_+$. The next result provides an explicit expression
for the post-retirement dual value function $\tilde G(\cdot)$, whose proof is similar to that of Lemma 6.1 in \cite{bo2026extended}, and hence we omit it here. 
\begin{lemma}\label{lem:structtildeG}
Let Assumption {\bf(A$_Z$)} and   $\left(\boldsymbol{A}_{\boldsymbol{\rho}}\right)$ hold. For any $(y, z) \in (0, \beta] \times \mathbb{R}_+$, we have
\begin{align}\label{eq:explicitformtildeG}
\tilde{G}(y, z) = 
\dfrac{(1 - p)^3}{p(\rho(1 - p) - \xi p)} y^{-\frac{p}{1 - p}} + \dfrac{(1 - p)^2}{\rho(1 - p) - \xi p} \beta^{-\frac{1}{1 - p}} y + z \left( y - \dfrac{\beta^{-(\kappa - 1)}}{\kappa} y^\kappa \right), 
\end{align}
where the constant $\kappa$ is the positive root of the quadratic equation $$\xi\kappa^2 + (\rho +\mu_X-\mu_Z - \xi)\kappa + \mu_Z - \rho = 0$$ with $\xi:=\mu^{\top}(\sigma \sigma^{\top})^{-1} \mu / 2$, which is given by
\begin{align*}
 \kappa = \frac{-(\rho +\mu_X-\mu_Z - \xi) + \sqrt{(\rho +\mu_X-\mu_Z- \xi)^2 + 4\xi(\rho -\mu_Z)}}{2\xi}\in(0,1). 
\end{align*}
\end{lemma}
Using the explicit form \eqref{eq:explicitformtildeG}, it can be shown that (i) $y \mapsto \tilde G(y,z)$ is strictly decreasing and strictly convex, and (ii) $z\mapsto \tilde G(y,z)$ is decreasing. We also have the following lemma, which is useful for solving the optimal stopping problem \eqref{eq:tildev}. 
\begin{lemma}\label{lem:growth}
Let Assumption {\bf(A$_Z$)} and  $\left(\boldsymbol{A}_{\boldsymbol{\rho}}\right)$ hold. Then, for any $(y,z)\in(0,\alpha]\times\R_+$, we have
\begin{align*}
\Ex\left[\int_0^{\infty} e^{-\rho t} \left|\tilde U(Y_t^\alpha)+(r-\mu_X) Z_t Y_t^\alpha-\ell Z_t+r_cY_t^\alpha-\ell_c\right|dt\right]<+\infty,
\end{align*}
and
\begin{align*}
\Ex\left[\sup_{t\geq 0}\left|
e^{-\rho t}\tilde G(Y_t^\alpha,Z_t)\right|\right]<+\infty.
\end{align*}
Here, the process $Y^{\alpha}=(Y^{\alpha}_t)_{t\geq0}$ is the RDP given by  \eqref{eq:Y_t} with $Y_0=y$, and $Z=(Z_t)_{t\geq0}$ is the benchmark process given by \eqref{eq:Zt} with $Z_0=z$.
\end{lemma}

Using the standard  theory of optimal stopping (see e.g., \citealt{shiryaev2008optimal}), we define the following continuation and stopping regions, respectively
\begin{align}\label{eq:continuous-stopping_region}
\mathcal{C} &:=\{ (y,z) \in (0,\alpha] \times \R_+;~\tilde v (y,z) > \tilde G(y,z)\},~~
\mathcal{S} :=\{ (y,z) \in (0,\alpha] \times \R_+;~\tilde v (y,z) = \tilde G(y,z)\}.
\end{align}
Using the integrability condition from Lemma \ref{lem:growth}, we can apply Theorem 3.3 in \cite{shiryaev2008optimal} to conclude that
\begin{align*}
\tau^*(y,z) = \inf \left\{ t\geq 0;~\left(Y_t^{ \alpha,y} ,Z_t^z \right) \in \mathcal{S} \right\},~~(y,z)\in(0,\alpha]\times\R_+
\end{align*}
is the optimal stopping time for problem \eqref{eq:tildev}.

\subsection{Structural Properties of Dual Value Function}

We now study the structural properties of the dual value function $\tilde{v}(y,z)$ defined by \eqref{eq:tildev}, which will be
used to characterize the optimal stopping region. The following lemma clarifies how the benchmark and the dual wealth state affect the dual value function, and they provide the monotonicity and convexity needed for the subsequent free-boundary analysis. 
\begin{lemma}\label{lem:struct}
Let Assumption {\bf(A$_Z$)} and  $\left(\boldsymbol{A}_{\boldsymbol{\rho}}\right)$ hold. Then, we have
\begin{itemize}
    \item[{\rm(i)}]for any $y\in(0,\alpha]$, the dual function $z \mapsto \tilde{v}(y, z)$ is decreasing; 
    \item[{\rm(ii)}] for any $z\in\R_+$, the dual function $y \mapsto \tilde{v}(y, z)$ is strictly decreasing and strictly  convex.
\end{itemize} 
\end{lemma}

The following lemma verifies the continuity of the dual value function, thereby establishing the continuous-pass condition for $\tilde v$  in Eq.~\eqref{eq:tildevGyBz}.  

\begin{lemma}\label{lem:continuous}
Let Assumption {\bf(A$_Z$)} and  $\left(\boldsymbol{A}_{\boldsymbol{\rho}}\right)$ hold. Then, the dual value function $\tilde{v}(y,z)$ is continuous in $(y,z)\in (0, \alpha] \times \R_+$.
\end{lemma}

We next provide the probabilistic expression of the first partial derivative of the value function with respect to 
$y$ in the following lemma, which ensures the smooth-fit condition of the dual value function $\tilde v$ in \eqref{eq:tildevGyBz}. 
\begin{lemma}\label{prop:tvy}
Let Assumptions $\left(\boldsymbol{A}_{\boldsymbol{Z}}\right)$ and $\left(\boldsymbol{A}_{\boldsymbol{\rho}}\right)$ hold. Introduce the $\mathbb{F}$-stopping time $\tau_f^{k}:=\inf\{t\geq 0;~(\rho-\xi) t+\sqrt{2\xi} W_t=k\}$ for $k\geq0$ and recall that $\xi=\mu^{\top}(\sigma \sigma^{\top})^{-1} \mu / 2$. Let $(y,z)\in (0,\alpha]\times \R_+$ and $\tau^*(y,z):=\inf\{ t\geq 0;~ \tilde v(Y_t^{\alpha,y},Z_t^z)=\tilde G(Y_t^{\alpha,y},Z_t^z)\}$. Then, we have
\begin{align}\label{eq:tildevy}
  \tilde v_y(y,z)&=\frac{1}{y}\Ex\Bigg[\int_0^{\tau^*(y,z)\wedge\tau_f^{-\ln\frac{y}{\alpha}}} e^{-\rho s}
    \left(\tilde U'\left(y \eta_s\right)+(r-\mu_X) Z_s^z+r_c\right)y \eta_s ds\nonumber\\
    &\qquad\quad+e^{-\rho \tau^*(y,z)}\tilde G_y\left(y \eta_{\tau^*(y,z)},Z_{\tau^*(y,z)}^z\right)y \eta_{\tau^*(y,z)}{\mathds{1}}_{\left\{\tau^*(y,z)<\tau_f^{-\ln\frac{y}{\alpha}}\right\}}\Bigg]
\end{align}
with the process 
$\eta_t:=e^{(\rho-\xi)t-\mu^\top (\sigma \sigma^\top)^{-1}\sigma W_t}$ for $t\geq 0$.
\end{lemma}


To characterize the stopping boundary of the dual problem \eqref{eq:tildev}, we
need to understand how the incremental value of delaying retirement varies with
the dual wealth state. The following lemma provides the monotonicity of the continuation premium profit  $y\mapsto\tilde v(y,z)-\tilde G(y,z)$ for $y\in(0,\alpha]$. 
\begin{lemma}\label{lem:derivativev_G}
 Let Assumptions $(\boldsymbol{A}_{\boldsymbol{Z}})$ and  $(\boldsymbol{A}_{\boldsymbol{\rho}})$ hold. Then, we have $\tilde v_y(y,z)\geq \tilde G_y(y,z)$ for all $(y,z)\in (0,\alpha]\times\R_+$.
\end{lemma}

The next lemma characterizes the structure of the continuation and stopping regions defined in \eqref{eq:continuous-stopping_region}.
It shows that, for each benchmark asset value \(z\), the decision to retire is governed by a
one-dimensional boundary in the dual wealth state \(y\). This threshold representation reduces the stopping problem to the analysis of the free boundary
\(y_B(\cdot)\).
\begin{lemma}\label{lem:sectio}
Let Assumptions $(\boldsymbol{A}_{\boldsymbol{Z}})$ and  $(\boldsymbol{A}_{\boldsymbol{\rho}})$ hold. Then, we have
\begin{align}\label{eq:continuous-stopping_region2}
\mathcal{C}=\left\{(y,z) \in (0,\alpha]\times\R_+;~ y>y_B(z)\right\},~~
\mathcal{S}=\left\{(y,z) \in (0,\alpha]\times\R_+;~ y\leq y_B(z)\right\}.
\end{align}
Here, the stopping boundary $y_B(z)$ for $z\in\R_+$ is defined by
\begin{align}\label{eq:stopingboundary}
y_B(z):=\sup\{y\in (0,\alpha];~\tilde v(y,z)=\tilde G(y,z)\}.
\end{align}
\end{lemma}


Let $O_{\alpha}:=(0,\alpha)\times \R_+$. Using the dynamic programming principle, we have the following variational inequality (VI) satisfied by the dual value function $\tilde{v}(y,z)\in W^{2,q}(O_{\alpha})$ with $q\geq 1$ (see, e.g., \citealt{friedman1982variational,peskir2006optimal}), for $(y,z)\in O_{\alpha}$,
 \begin{align}\label{eq:vihatv0}
\begin{cases}
 \displaystyle  \tilde {\cal L} \tilde {v}+(r-\mu_X) y z+r_cy+\tilde U(y)-\ell z-\ell_c=0,~~\text { if } \tilde{v}>\tilde{G};\\[0.5em]
\displaystyle  \tilde {\cal L} \tilde {v}+(r-\mu_X) y z+r_c y+\tilde U(y)-\ell z-\ell_c\leq 0,~~ \text { if } \tilde {v}=\tilde {G}
; \\[0.5em] 
\displaystyle \tilde {v}_y(\alpha,z)=0,\quad \forall z \in\R_+,
\end{cases}
\end{align}
where the operator $\tilde{\cal L}$ acted on $C^2(O_{\alpha})$ is defined by, for $\varphi\in C^2(O_{\alpha})$,
\begin{align}\label{eq:tildeL}
 \tilde {\cal L}\varphi:=&  {\mu^{\top}(\sigma \sigma^{\top})^{-1} \mu \over2} y^2\partial_{yy}\varphi
 +{\sigma^2_Z z^2|\gamma|^2\over2}\partial_{zz}\varphi
 -  \mu^{\top}(\sigma\sigma^{\top})^{-1}\sigma\sigma_Z z\gamma  y\partial_{yz}\varphi
 +\rho y\partial_y\varphi\notag\\
 &+\mu_Z z\partial_z\varphi-\rho \varphi.
\end{align}
By virtue of Lemma~\ref{prop:tvy}, the smooth-fit condition of the dual value function $\tilde v$ holds, i.e., $\tilde v_y(y_B(z),z)=\tilde G_y(y_B(z),z)$, because that when $y=y_B(z)$, we have $\tau^*(y_B(z),z)=0$ for $z\in \R_+$. 
Then, the VI \eqref{eq:vihatv0} can be rewritten as the following free-boundary problem, for $z\in \R_+$,
\begin{align}\label{eq:tildevGyBz}
\begin{cases}
\displaystyle \tilde {\cal L} \tilde {v}+(r-\mu_X) y z+r_c y+\tilde U(y)-\ell z-\ell_c=0,\\[0.6em]
\displaystyle \tilde v(y_B(z),z)=\tilde G(y_B(z),z)\quad \quad\text{(continuous\ pass)},\\[0.6em]
\tilde v_y(y_B(z),z)=\tilde G_y(y_B(z),z)\quad \text{(smooth-fit)},\\[0.6em] 
\displaystyle \tilde {v}_y(\alpha,z)=0\quad \text{(Neumann boundary)}. 
\end{cases}
\end{align}
Furthermore, Lemma~\ref{lem:derivativev_G} yields that $\tilde v_y$ satisfies the following PDE, on $O_{\alpha}$,
 \begin{align} \label{eq:vihatv}
\begin{cases}
 \displaystyle \tilde {\cal A} u+(r-\mu_X)z+r_c +\tilde U'(y)=0,\\[0.4em]
 \displaystyle u(y_B(z),z)=\tilde G_y(y_B(z),z),\quad \forall z\in\R_+, \\[0.4em]
 \displaystyle u(\alpha,z)=0,\quad \forall z\in\R_+
\end{cases}
\end{align}
with the operator $\tilde{\cal A}$ acted on $C^2(O_{\alpha})$ being given by, for $\phi\in C^2(O_{\alpha})$, 
\begin{align}\label{eq:tildeA}
 \tilde {\cal A}\phi &:=  {\mu^{\top}\left(\sigma \sigma^{\top}\right)^{-1} {\mu \over 2}} (2y \partial_y+y^2\partial_{yy})\phi
 +{\sigma^2_Z z^2|\gamma|^2\over 2}\partial_{zz}\phi
 -  \mu^{\top}(\sigma\sigma^{\top})^{-1}\sigma\sigma_Z z\gamma  (\partial_z+y\partial_{yz})\phi
 +\rho y\partial_y\phi\nonumber\\
 &\quad+\mu_Z z\partial_z\phi. 
\end{align}
 By applying the standard theory of optimal stopping (e.g. Theorem 2.7.7 in \citealt{karatzas1998methods} and Theorem 3.2 in \citealt{friedman1982variational}) together with Lemma~\ref{lem:struct}, Eq.~\eqref{eq:vihatv} has a unique solution $u\in C^{2}(\tilde{O}_{\alpha})$ with $\tilde{O}_{\alpha}:=\{(y,z)\in(0,\alpha);~ y_B(z)<y<\alpha,~z\in \R_+\}$ and $z\mapsto y_B(z)$ given by \eqref{eq:stopingboundary}. Therefore, the dual value function $\tilde v(\cdot)$  is $C^2(\mathcal{C})$ with $\tilde{v}_y(\cdot)\in C^2(\tilde{O}_{\alpha})$.

\section{Duality Theorem and Optimal Strategy}\label{sec:dualityTH}
We now develop the convex duality theorem for the hybrid stochastic control
problem \eqref{eq:u} and characterize the corresponding optimal stopping time and control strategy. We first present the main duality theorem, which establishes the fundamental dual relationship between the primal value function $v(x, z)$ and the dual function $\tilde{v}(y, z)$, and then provides feedback expressions for the optimal controls in terms of the reflected dual process. This result converts the original control-stopping problem
with state reflection into a dual optimal stopping problem and provides the bridge back to the primal formulation.
\begin{theorem}[Duality Theorem]\label{th:duality} Let Assumptions $\left(\boldsymbol{A}_{\boldsymbol{Z}}\right)$ and $\left(\boldsymbol{A}_{\boldsymbol{\rho}}\right)$ hold. Then, we have, for any $(x, z) \in$ $\mathbb{R}_{+}^2$,
\begin{align*}
    v(x, z)=\inf_{y \in(0, \alpha]}(\tilde{v}(y, z)+x y).
\end{align*}
Furthermore, introduce the following triplet of control strategy $(\tau^*,\theta_t^*, c_t^*)$ with $0\leq t <\tau^*$ by
\begin{align}\label{eq:optimalcontr}
\begin{cases}
\displaystyle \tau^*=\inf\{t\geq 0;~\tilde v (Y_t^{\alpha,*},Z_t)=\tilde G(Y_t^{\alpha,*},Z_t)\},\\[0.4em]
\displaystyle \theta_t^*=\left(\sigma \sigma^{\top}\right)^{-1}\left(\mu Y_t^{\alpha,*} \tilde{v}_{y y}\left(Y_t^{\alpha,*}, Z_t\right)+\sigma_Z Z_t  \sigma \gamma \tilde{v}_{y z}\left(Y_t^{\alpha,*}, Z_t\right)-\sigma_Z Z_t  \sigma \gamma\right), \\[0.4em]
\displaystyle c_t^*=I_U\left(Y_t^{\alpha,*}\right),
\end{cases}
\end{align}
and for $t\geq \tau^*$,
\begin{align}\label{eq:optimalcontr2}
    \begin{cases}
\displaystyle  \theta_t^*=\left(\sigma \sigma^{\top}\right)^{-1}\left(\mu Y_t^{\tau,\beta,*} \tilde{G}_{y y}\left(Y_t^{{\tau,\beta,*}}, Z_t\right)+\sigma_Z Z_t  \sigma \gamma \tilde{G}_{y z}\left(Y_t^{{\tau,\beta,*}}, Z_t\right)-\sigma_Z Z_t  \sigma \gamma\right), \\[0.4em]
\displaystyle c_t^*=I_U\left(Y_t^{{\tau,\beta,*}}\right).
\end{cases}
\end{align}
Here, $I_U(\cdot)$ is the inverse function of $U^{\prime}(\cdot)$, the RDP $Y^{{\alpha,*}}=(Y_t^{{\alpha,*}})_{t \geq 0}$ is given by \eqref{eq:Y_t} with $Y_0=y^*(x, z)$ which is the function determined by $-\tilde{v}_y\left(y^*(x, z), z\right)=x$, and the RDP  $Y^{{\tau,\beta,*}}=(Y_t^{{\tau,\beta,*}})_{t \geq 0}$ is given by \eqref{eq:Y_tbeta} with $Y^{\tau,\beta}_{\tau^*}=Y_{\tau^*}^{\alpha}$. Then, we have $(\tau^*,\theta^*, c^*)=(\tau^*,(\theta_t^*, c_t^*)_{t\geq0})\in{\mathbb{T}\times}\mathbb{U}^{\rm r}$ is an optimal control pair for problem~\eqref{eq:u}, that is, for any $(x,z)\in\R_+^2$,
\begin{align*}
v(x, z)=\sup _{(\tau,\theta, c) \in \mathbb T\times \mathbb{U}^{\rm{r}}} J(x, z ; \tau, \theta, c)=J\left(x, z ; \tau^*,\theta^*, c^*\right),  
\end{align*}
where the original objective functional $J(\cdot)$ is defined in \eqref{eq:u}. 
\end{theorem}

Theorem~\ref{th:duality} gives the optimal strategy in terms of the dual function $\tilde v(y,z)$. The following corollary characterizes the optimal strategy in feedback form via the original value function $v(x,z)$.
\begin{corollary}\label{coro:originalstrategy}
Let Assumptions $\left(\boldsymbol{A}_{\boldsymbol{Z}}\right)$ and $\left(\boldsymbol{A}_{\boldsymbol{\rho}}\right)$ hold. Define the following optimal feedback control functions by, for any $(x, z) \in \mathbb{R}_{+}^2$, 
 \begin{align}\label{eq:x_Bz}
 \begin{cases}
\displaystyle \theta^*(x, z):=\begin{cases}
 \displaystyle   -\left(\sigma \sigma^{\top}\right)^{-1} \frac{\mu v_x(x, z)+\sigma_Z z \sigma \gamma\left(v_{x z}(x, z)-v_{x x}(x, z)\right)}{v_{x x}(x, z)},~\text{ if } 0<x<x_B(z),\\[0.6em]
 \displaystyle   -\left(\sigma \sigma^{\top}\right)^{-1} \frac{\mu G_x(x, z)+\sigma_Z z \sigma \gamma\left(G_{x z}(x, z)-G_{x x}(x, z)\right)}{G_{x x}(x, z)},~\text{ if } x\geq x_B(z);\\
\end{cases}\\[0.6em]
\displaystyle c^*(x, z):={v_x}^{\frac{1}{p-1}}(x, z);\\[0.6em]
\displaystyle \tau^*(x,z):=\inf \{t\geq 0;~X^{*,x}_t\geq x_B(Z_t^z)\},
\end{cases}
\end{align}
where $x_B(z):=-\tilde v_y(y_B(z),z)$ for $z\in \R_+$ and $   G(x,z):= \inf_{y\in (0,\beta]}(\tilde G(y,z)+xy)$.
Consider the controlled state process $(X^*,Z)=(X_t^*,Z_t)_{t \geq 0}$ with feedback controls $\theta^*=(\theta^*(X_t^*,Z_t))_{t \geq 0}$ and $c^*=(c^*(X_t^*,Z_t))_{t \geq 0}$. 
Then, the triplet $(\tau^*,(\theta_t^*, c_t^*)_{t\geq0}) \in \mathbb{T}\times\mathbb{U}^{\rm r}$ is an optimal investment-consumption strategy. That is, for any admissible control $(\tau,\theta,c) \in \mathbb{T}\times\mathbb{U}^{\rm r}$, we have
\begin{align}\label{eq:condBC}
&\Ex\left[ \int_0^{\tau } e^{-\rho t} (U(c_t )-\ell Z_t-\ell_c)dt-\alpha\int_0^{\tau} e^{-\rho t}dL_t^{X }+\int_{\tau }^\infty e^{-\rho t} U(c_t )dt- \beta \int_\tau^{\infty} e^{-\rho t}dL_t^{X }\right]\nonumber\\
&\qquad\leq v(x,z), \quad \forall (x,z) \in \mathbb{R}_+^2,    
\end{align}
where the above equality holds when one takes $(\tau,\theta,c)=(\tau^*,\theta^*,c^*)$.
\end{corollary}

Theorem~\ref{th:duality} and Corollary~\ref{coro:originalstrategy} establish the
connection between the primal control-stopping problem \eqref{eq:u} and the dual
optimal stopping problem \eqref{eq:tildev}. They show that the dual free
boundary determines the primal retirement threshold, while the derivatives of
the primal value function generate the optimal consumption and portfolio rules.
Then, using the duality relationship and Lemma~\ref{lem:struct}, we can easily obtain the structural properties of the original problem, and hence the proof of the following lemma is omitted.
\begin{lemma}[Structure properties of original value function $v$]
Let Assumptions $\left(\boldsymbol{A}_{\boldsymbol{Z}}\right)$ and $\left(\boldsymbol{A}_{\boldsymbol{\rho}}\right)$ hold. Then, the original value function $v(x,z)$ satisfies the following properties:
\begin{itemize}
\item [{\rm(i)}] for any $z \in \R_+$, the value function $x \mapsto {v}(y, z)$ is strictly increasing and strictly concave;
\item [{\rm(ii)}] the value function ${v}(x, z)$ is continuous in $(x,z)\in \R_+^2$;
\item[{\rm(iii)}] for any $x > 0$, the value function $z \mapsto {v}(x, z)$ is increasing.
\end{itemize}  
\end{lemma}
\section{Explicit Examples}\label{sec:exa}
This section presents a sensitivity analysis for the original optimal consumption and retirement time problem \eqref{eq:u}. From the definition of the value function $v$, it is straightforward to verify that higher income, lower leisure preference, and lower shortfall risk cost lead to a higher expected utility for the investor. To further study how these parameters affect the optimal retirement threshold, the portfolio-consumption strategy and the expected largest shortfall risk, we provide two explicit examples and analyze the investor's benefit and the expected largest shortfall risk from the retirement option.

\subsection{ Example with Zero Benchmark} 
The first example provides a special case with the benchmark process  $Z_t\equiv 0$ for all $t\geq0$.  In this case, labor income and leisure preference reduce to the constant rate $(r_t,\ell_t)=(r_c,\ell_c)$, and the shortfall process becomes
$ A_t^{(\theta,c,\tau)}
    =0\vee \sup_{s\leq t}\left(-V_s^{\theta,c,\tau}\right),
$
which records the running maximum deficit of the investor's wealth below zero. Hence, the shortfall penalty captures the investor's aversion to borrowing, insolvency, or persistent negative wealth. 
This benchmark-free case can be viewed as an extension of the classical Merton consumption-investment problem \citep{merton1971optimum}, in which the investor also chooses a retirement time and faces a penalty for the running maximum negative wealth. It is also related to discretionary-stopping formulations of utility maximization, such as \citet{karatzas2000utility}, but differs by incorporating shortfall-risk penalization into the joint portfolio, consumption, and retirement-timing problem.
This explicit case serves as a natural baseline for evaluating the additional benchmark-tracking effects introduced when $Z$ is stochastic in Section~\ref{sec:ex2}.

\begin{lemma}\label{lem:ex1}
When the benchmark process $Z=(Z_t)_{t\geq0}\equiv 0$ (i.e., $\mu_Z=\sigma_Z=0$),   the solution of VI~\eqref{eq:tildev} has the following closed-form given by
\begin{align*}
\tilde v(y,0)=\tilde G(y,0)+P(y),\quad \forall y\in(0,\alpha],
\end{align*}
where the function $P(y)$ for $y\in (0,\alpha]$ and the stopping boundary have the following explicit representation: 
\begin{itemize}
\item [{\rm(i)}] if $\alpha \leq \frac{\ell_c}{r_c}$, then $P(y)\equiv 0$,  $y_B(0)= \alpha$ and $x_B(0)=0$;
\item [{\rm(ii)}] if $\alpha> \frac{\ell_c}{r_c}$, then we have, for any $y\in(0,\alpha]$,
\begin{align*}
P(y)=\left(C_1 y+C_2{y}^{-\frac{\rho}{\xi}}-\frac{r_c}{\rho+\xi}y\ln y-\frac{\ell_c}{\rho}\right)\mathds 1_{\{y>y_B\}},
\end{align*}
where the constants $C_1=\frac{r_c(1+\ln y_B)}{\rho+\xi}+\frac{\ell_c}{\rho+\xi}y_B^{-1}-\frac{\rho r_c}{(\rho+\xi)^2}$ and $C_2=\frac{\ell_c\xi/\rho-r_c\xi/(\rho+\xi)y_B}{\rho+\xi}y_B^{\frac{\rho}{\xi}}$ with $y_B{\in (0,\frac{\ell_c}{r_c})}$ being the unique solution of the following equation:
\begin{equation*}
	\left(\frac{\ell_c/r_c }{y_B}-\frac{\rho }{\rho+\xi}\right)\left[1- \left(\frac{y_B}{\alpha}\right)^{\frac{\rho}{\xi}+1}\right]+\ln\left(\frac{y_B}{\alpha}\right) = -\tilde G_y(\alpha,0)\frac{\rho+\xi}{r_c}.
\end{equation*}
Furthermore, the optimal retirement boundary for the original problem is given by
\begin{align*}
x_B(0)=-\tilde G_y(y_B,0)=\frac{(1-p)^2\left(y_B^{-\frac{1}{1-p}}-\beta^{-\frac{1}{1-p}}\right)}{\rho (1-p)-\xi p},
\end{align*}
and the optimal portfolio-consumption strategy satisfies
\begin{align*}
    \theta^*(x,0)=-(\sigma \sigma^\top)^{-1}\left(
    \frac{\mu (p-1)}{\rho (p-1)+\xi p}(v_x)^{\frac{1}{p-1}}+\mu\left(
    \frac{\ell_c}{\xi}-\frac{r_c \rho}{\xi(\rho+\xi)}y_B
    \right)y_B^{\frac{\rho}{\xi}} (v_x)^{-\frac{\rho+\xi}{\xi}}-\frac{\mu r_c}{\rho+\xi}
    \right).
\end{align*}
\end{itemize} 

\end{lemma}

To characterize the individual's benefit from the retirement option, we also define a portfolio-consumption problem without retirement, that is, for $(x,z)\in\mathbb{R}_+$,
\begin{align}\label{eq:uworking}
v^{\rm w}(x,z)
:= \sup_{(\theta,c)\in{\mathbb{U}^r}}\Ex\left[ \int_0^\infty e^{-\rho t} (U(c_t)-\ell Z_t-\ell_c )dt- \alpha \int_0^\infty e^{-\rho t}dL_t^X\right].
\end{align}
Then, using the static budget constraint similar to Lemma~\ref{lem:XY}, it is not difficult for us to derive that $v^{\rm w}(x,z)=\inf_{y\in(0,\alpha]}(\tilde v^{\rm w}(y,z)+xy)$, where the duality function $\tilde v^{\rm w}(y,z)$ is given by
\begin{align*}
\tilde v^{\rm w}(y,z)=\Ex\left[\int_0^\infty e^{-\rho t}\left(\tilde U(Y_t^\alpha)+(r-\mu_X)Z_tY_t^\alpha-\ell Z_t+r_c Y_t^\alpha -\ell_c\right) dt\right].
\end{align*}
Using It\^o's formula, we have the following explicit solution for the individual without considering the retirement option:
\begin{align*}
\tilde v^{\rm w} (y,z)&=\left(
\frac{r_c(\ln\alpha+1)}{\rho+\xi}
+\frac{1-p}{\rho+\frac{\xi p}{p-1}}\alpha^{\frac{1}{p-1}}
\right)y
-\frac{\ell_c}{\rho}
-\frac{r_c}{\rho+\xi}y\ln y
+\frac{(1-p)^2}{p\left(\rho+\frac{\xi p}{p-1}\right)}y^{\frac{p}{p-1}}\\
&\quad +z\left((1-\frac{r}{\mu_X})y+\frac{r-\mu_X}{\kappa \mu_X}\alpha^{1-\kappa}y^\kappa+\frac{\ell}{\mu_Z-\rho}\right),
\end{align*}
and 
\begin{align*}
\theta^{\rm w}(x,z)&=-(\sigma\sigma^\top)^{-1}\Bigg[-
\frac{\mu r_c}{\rho+\xi}+\frac{\mu (p-1)}{\rho (p-1)+\xi p}(v^{\rm w}_x)^{\frac{1}{p-1}}-\frac{r}{\mu_X}+\mu z \left(\frac{r}{\mu_X}-1\right)(\kappa-1)(v^{\rm w}_x /\alpha)^{\kappa-1}\\
&\qquad+\sigma \sigma_Z \gamma z\left(\frac{r-\mu_X}{\mu_X}(v^{\rm w}_x/\alpha)^{\kappa -1}\right)\Bigg],\\
c^{\rm w}(y,z)&=(v^{\rm w}_x)^{\frac{1}{p-1}}.
\end{align*}

\begin{figure}[!htbp]
    \centering
    \begin{subfigure}[b]{0.48\textwidth}
        \centering
        \includegraphics[width=\textwidth]{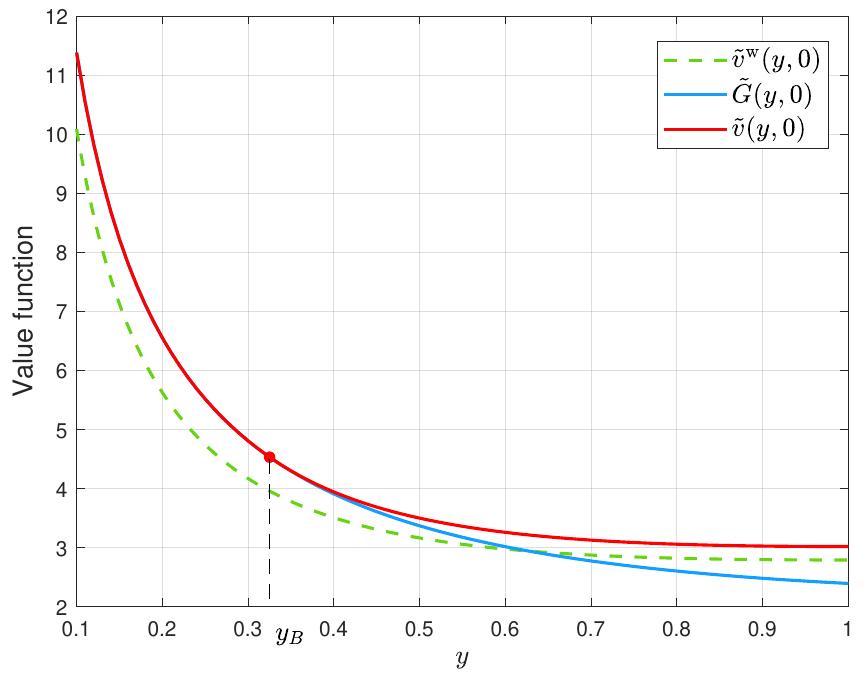} 
    \end{subfigure}
    \hfill  
    \begin{subfigure}[b]{0.48\textwidth}
        \centering
    \includegraphics[width=\textwidth]{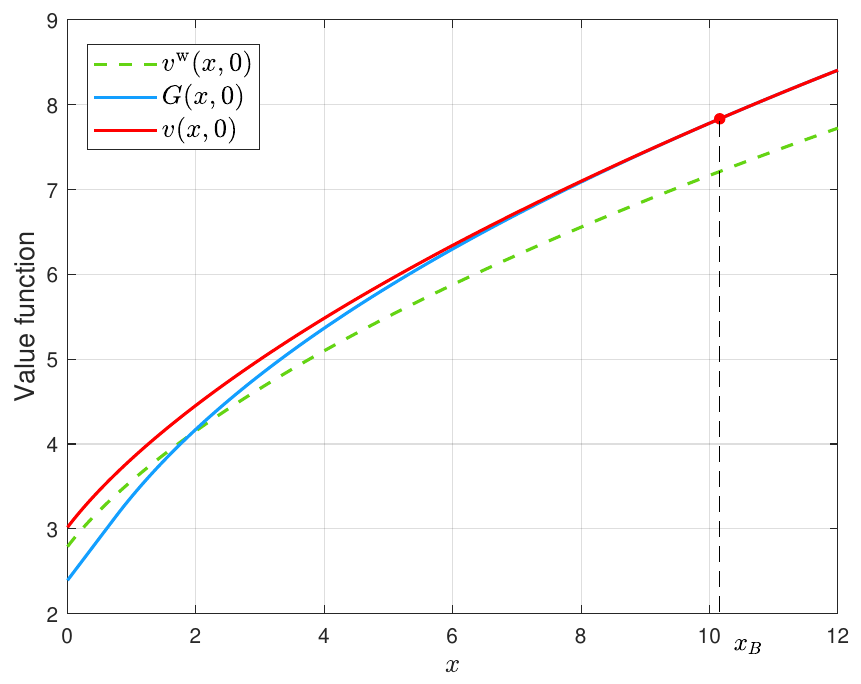}  
    \end{subfigure} 

    \caption{Comparison of value functions $\tilde{v}(y,z)$, $\tilde{G}(y,z)$, $\tilde v^{\rm w}$, ${v}(y,z)$, ${G}(y,z)$, and $v^{\rm w}(x,z)$. The baseline parameters are set as $z=0, \rho = 0.8$, $\alpha =1$, $\beta = 1.5$, $p = 0.45$, $\sigma = 0.1$, $\gamma = 1$, $\mu = 0.1$, $r_c=2$ and $\ell_c=1.5$.}
    \label{fig:ex1_values}
\end{figure}

Figure~\ref{fig:ex1_values}  compares value functions across the original problem~\eqref{eq:u}, the post-retirement portfolio problem~\eqref{eq:tildeG}, and the no-retirement portfolio-consumption problem~\eqref{eq:uworking}, in both the dual space (left panel) and the primal space (right panel).  It is observed that the dual value function $\tilde{v}$ dominates both $\tilde{G}$ and $\tilde{v}^{\rm w}$ over the entire state space and the primal value function $v$ dominates both $G$ and $v^{\rm w}$, which is consistent with the convex duality theory developed in Section 4. This directly quantifies the value of the early retirement option: investors may cease working and retire when their assets grow to a sufficient level to achieve greater expected utility. The gap between $v$ and $v^{\rm w}$ widens as wealth increases, which reflects that the early retirement option becomes more valuable at higher wealth levels; a pattern also observed in early retirement studies such as \cite{jin2006disutility} and \cite{dybvig2010lifetime}.
Conversely, the difference between $v$ and $G$ narrows as wealth approaches the optimal retirement boundary $x_B$. This gap represents the value of the option to continue working and delay retirement, which is strictly positive for all $x<x_B$, indicating that it is optimal to remain in the labor state when wealth is below the threshold. As wealth accumulates towards $x_B$, the marginal benefit of continuing to work diminishes, and the value of the continuation option vanishes at the boundary, where the value-matching and smooth-pasting conditions are satisfied (see also \citealt{karatzas1998methods,peskir2006optimal}). At this point, the investor is indifferent between retiring immediately and continuing to work, thereby triggering retirement optimally.

\begin{figure}[!h]
    \centering
    \includegraphics[width=0.95\linewidth]{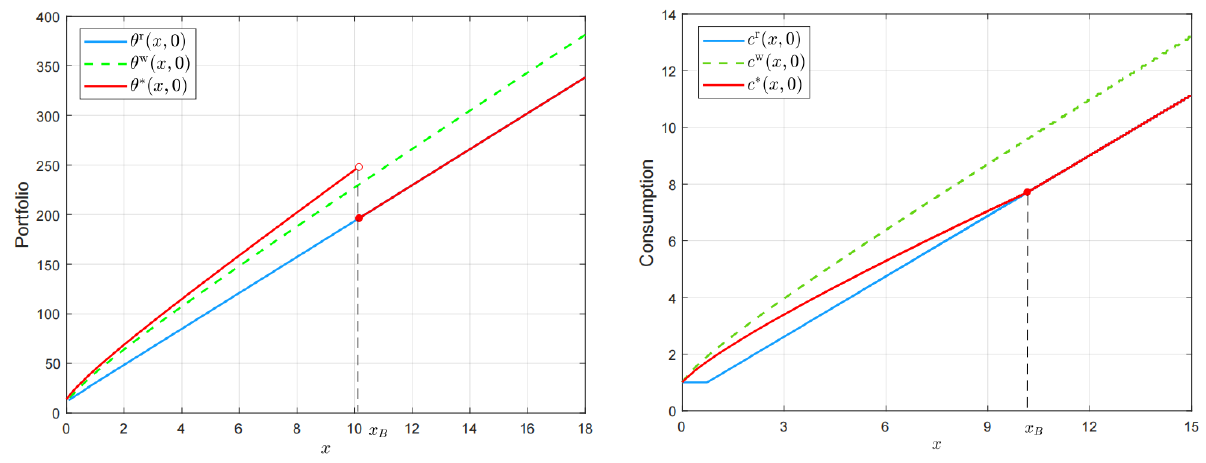}
    \caption{Comparison of portfolio-consumption strategies $\theta^*(x,z)$, $\theta^{\rm r}(x,z)$, $\theta^{\rm w}(x,z)$, $c^*(x,z)$, $c^{\rm r}(x,z)$, and $c^{\rm w}(x,z)$. The baseline parameters are set as $z=0, \rho = 0.8$, $\alpha =1$, $\beta = 1.5$, $p = 0.45$, $\sigma = 0.1$, $\gamma = 1$, $\mu = 0.1$, $r_c=2$ and $\ell_c=1.5$. }
    \label{fig:ex1_theta_c}
\end{figure}
Figure~\ref{fig:ex1_theta_c} compares the optimal portfolio and consumption strategies across the three scenarios in Figure~\ref{fig:ex1_values}.\footnote{The function $(\theta^{\rm r}(x,z),c^{\rm r}(x,z))$ is the portfolio-consumption strategy for post-retirement problem given in Corollary~\ref{coro:originalstrategy}, that is, for $(x,z)\in\R^2_+$,
$$\theta^{\rm r}(x,z) := -(\sigma\sigma^\top)^{-1} \frac{\mu G_x(x,z) + \sigma_Z z \sigma \gamma \big(G_{xz}(x,z) - G_{xx}(x,z)\big)}{G_{xx}(x,z)}, \quad
c^{\rm r}(x,z) := G_x^{\frac{1}{p-1}}(x,z).$$
}
Turning first to the portfolio strategy in the left panel, the behavior of optimal risk-taking differs fundamentally from the monotonic, wealth-increasing pattern in the frictionless benchmark in \cite{merton1969lifetime,merton1971optimum}. Instead, the optimal portfolio policy exhibits a clear two-stage structure across the retirement boundary. 
For $x<x_B$, the optimal pre-retirement portfolio $\theta^*$ lies higher than  the post-retirement portfolio $\theta^{\rm r}$ and the no-retirement benchmark $\theta^{\rm w}$.  More importantly, as wealth crosses the optimal retirement threshold \(x_B\), the optimal risky share drops discontinuously and remains significantly lower thereafter, a pronounced decline that is absent in the standard Merton framework. 
This ordering arises from two competing effects. First,  retirement is a risk-taking incentive. The presence of the early retirement option introduces a form of upside potential that encourages risk-taking. By increasing their exposure to risky assets, the investor accelerates wealth accumulation, thereby bringing the optimal retirement boundary $x_B$ closer in expectation. The option to exit the labor state once the wealth target is met effectively mitigates the perceived downside risk of aggressive investment.  Second, pre-retirement labor income is a risk buffer. Unlike the retiree, who must fund all future consumption from accumulated wealth, the working investor has future labor income flows that act as a ``human capital buffer". 
This buffer effect is consistent with \citet{BodieMertonSamuelson1992}, who show that labor-supply flexibility and future wage income can induce investors to take on more financial risk before retirement. 
The right panel illustrates that investors' consumption strategy increases with wealth. The no-retirement strategy \(c^{\rm w}\) yields the highest consumption level, as investors face only the lower pre-retirement shortfall costs. The post-retirement policy \(c^{\rm r}\) is the most restrained, driven by higher post-retirement shortfall costs and the absence of future labor income. Our optimal consumption \(c^*\) lies between the two, with a clear two-stage pattern. Before the retirement threshold, investors benefit from ongoing labor income and face lower shortfall costs, enabling them to consume more than if they retired immediately. 

\begin{figure}[!h]
    \centering
    \includegraphics[width=1\linewidth]{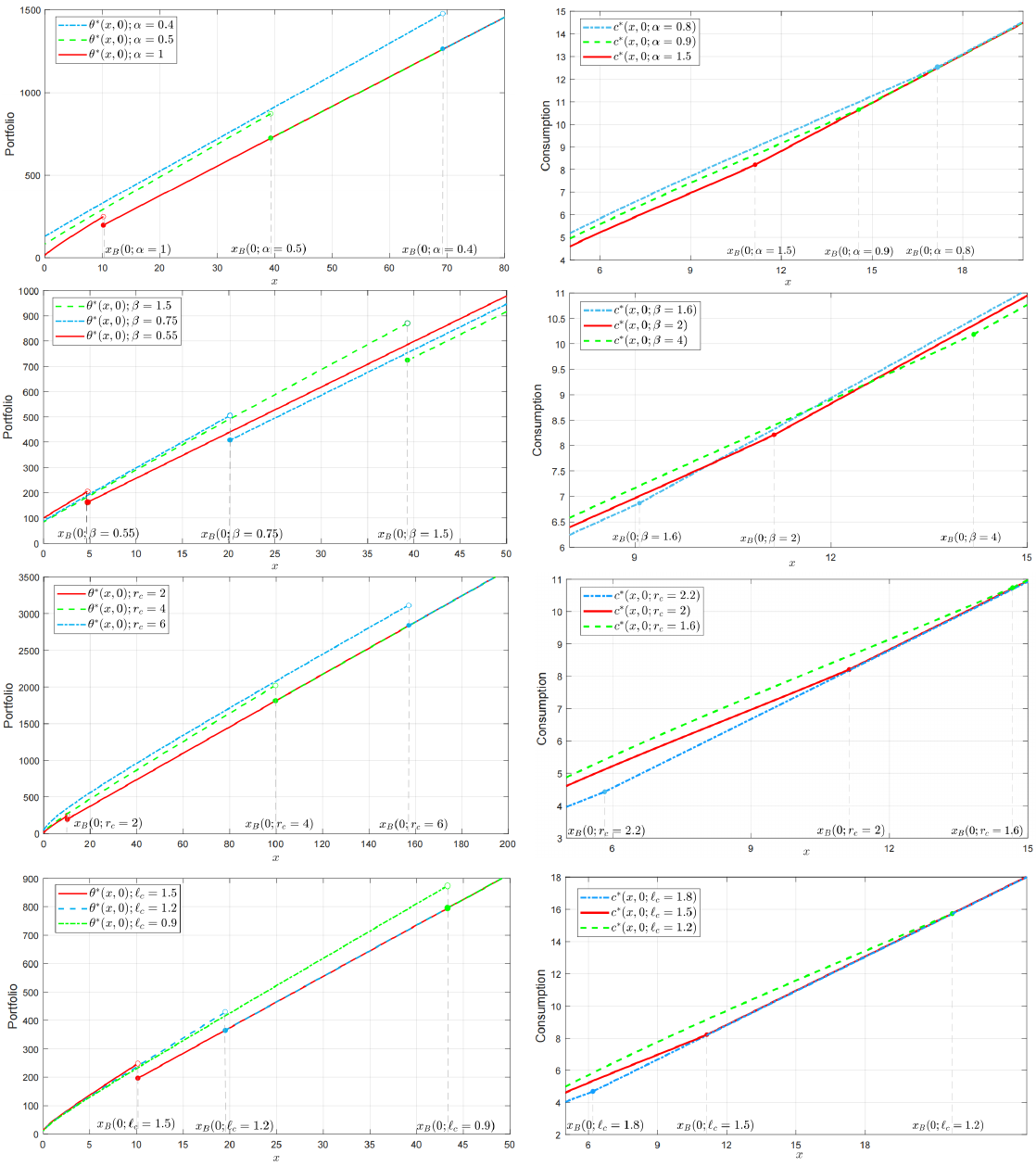}
    \caption{Comparison of portfolio-consumption strategies $\theta^*(x,z)$ and $c^*(x,z)$ under different parameters $(\alpha,\beta,r_c,\ell_c)$. The baseline parameters are set as 
    $z=0, \rho = 0.8, p = 0.45,\sigma = 0.1,\gamma = 1,\mu = 0.1$.}
    \label{fig:ex1_theta_c_sens}
\end{figure}
Figure~\ref{fig:ex1_theta_c_sens}  plots the optimal portfolio and consumption strategies across alternative parameter configurations, where the kink of each curve corresponds to the endogenous retirement wealth threshold $x_B$. A higher pre-retirement wealth shortfall risk management cost $\alpha$ amplifies investors’ downside risk aversion, reduces risky asset exposure, and raises the required precautionary wealth buffer. This pushes the endogenous retirement threshold $x_B$ upward and depresses current consumption to hedge against potential wealth shortages over the working period. An increase in the post-retirement wealth shortfall risk management cost $\beta$ motivates more aggressive risky investment during working life to accelerate wealth accumulation for covering future post-retirement shortfall risks, accompanied by a higher retirement threshold and restrained consumption after retirement. A larger labor income level $r_c$ enhances the human capital cushion, improves risk tolerance, promotes risky asset allocation, elevates lifetime resource endowment and the targeted retirement wealth level, and meanwhile increases optimal consumption in line with the permanent income hypothesis. A higher leisure utility level $\ell_c$ strengthens the appeal of retirement, lowers the critical retirement threshold $x_B$, induces more conservative portfolio strategies, and moderately raises current consumption to balance lifetime labor and leisure trade-offs.
Overall, higher shortfall risk costs amplify precautionary saving and risk avoidance, while income endowment and leisure preference shape the life-cycle equilibrium by altering budget constraints and utility trade-offs.

To further characterize the effects of retirement options on the investors' shortfall risk, the following corollary provides the expected largest shortfall risk measure before and after retirement.
\begin{corollary}\label{coro:ELS}
The expected largest shortfall risk (ELS) is given by
\begin{align*}
&\Ex\left[\alpha\int_0^{\tau^*}
e^{-\rho t} dA_t^{\theta^*,c^*}+\beta\int_{\tau^*}^\infty
e^{-\rho t} dA_t^{\theta^*,c^*}\right]=\left(
\frac{\rho}{\xi}K-\frac{(p-1){\alpha^{\frac{1}{p-1}}}}{\rho(1-p)-\xi p}\right) v_x(x,0)+ \alpha ^{{\rho \over \xi}+1}K (v_x(x,0))^{-\frac{\rho}{\xi}}\\
&\quad+\frac{(p-1)^2}{p[\rho(1-p)-\xi p]}(v_x(x,0))^{\frac{p}{p-1}}-\frac{\ell_c}{\rho}+\frac{\hat{g}(y_B)}{\rho y_B + \xi \alpha^{\frac{\rho}{\xi}+1} y_B^{-\frac{\rho}{\xi}}} \left(\rho v_x(x,0) + \xi \alpha^{\frac{\rho}{\xi}+1} (v_x(x,0))^{-\frac{\rho}{\xi}}\right)-v(x,0),
\end{align*}
where the constant $K$ is defined by
\begin{align*}
K:=\frac{\xi}{\xi y_B^{-\rho/\xi}\alpha^{ \frac{\rho}{\xi}+1}+\rho   y_B}\left(\frac{\ell_c}{\rho }-\frac{(p-1)^2(y_B^{\frac{1}{p-1}}-\frac{p}{p-1}\alpha^{\frac{1}{p-1}})}{p[\rho(1-p)-\xi p]} y_B\right).
\end{align*}
\end{corollary}

\begin{figure}[!h]
    \centering
    \includegraphics[width=0.95\linewidth]{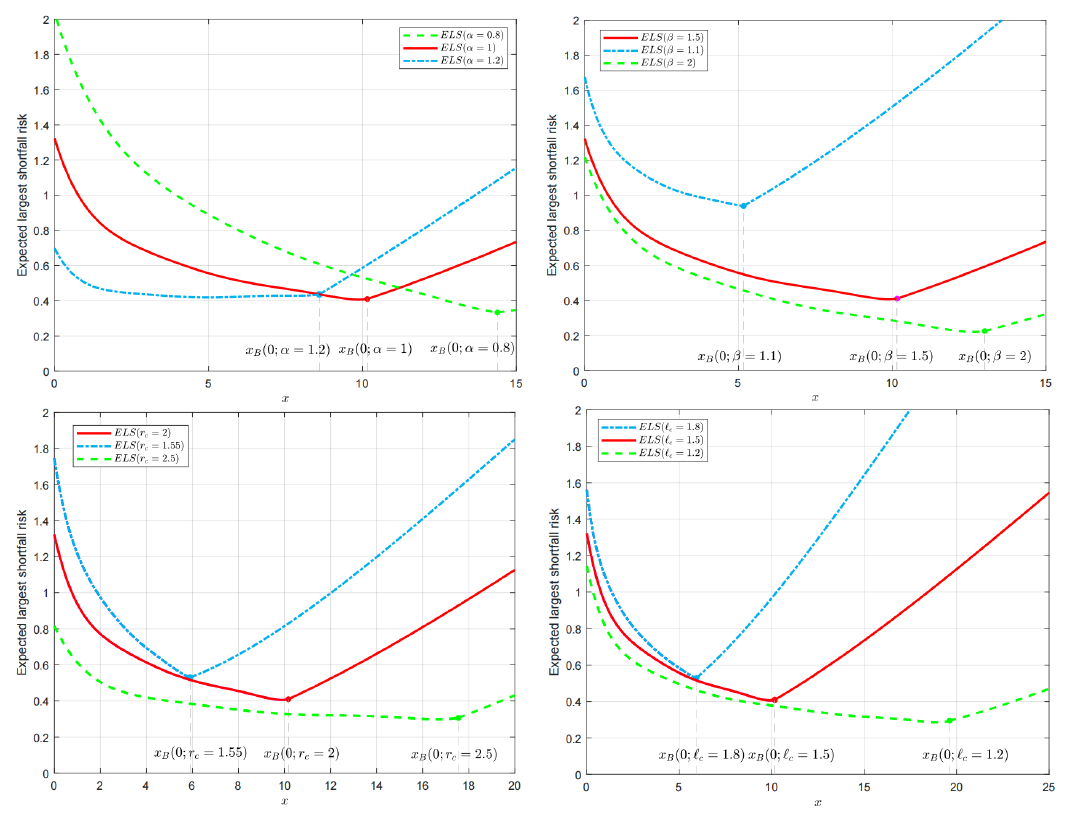}
    \caption{The expected largest shortfall risk under different parameters $(\alpha,\beta,r_c,\ell_c)$. The baseline parameters are set as $z=0, \rho = 0.8$, $p = 0.45$, $\sigma = 0.1$, $\gamma = 1$, $\mu = 0.1$. 
    }
    \label{fig:ELSrisk}
\end{figure}
Figure~\ref{fig:ELSrisk} illustrates the expected largest shortfall (ELS)  risk across different wealth levels under varying parameters. Consistent with the risk-mitigation effect of wealth during working life, the ELS exhibits a U-shaped pattern: it declines with wealth for \(x < x_B\) as accumulated assets buffer downside risk, and rises for \(x \ge x_B\) when consumption must be fully financed from wealth, exposing a larger asset pool to market volatility.  
This non-monotonicity distinguishes our framework from classical models (\citealt{farhi2007saving}), where standard lifecycle trade-offs govern retirement timing. We explore an additional channel: the retirement boundary marks a structural shift in downside-risk exposure, as labor income mitigates risk only during working life. The comparative statics further illustrate this contribution.  
 Higher pre-retirement (\(\alpha\)) and post-retirement (\(\beta\)) shortfall risk costs shift the ELS profile upward and delay retirement, as stronger precautionary motives encourage greater wealth accumulation to meet elevated risk penalties. Conversely, higher labor income (\(r_c\)) and leisure utility (\(\ell_c\)) reduce overall ELS levels and accelerate retirement, since labor income provides an effective risk hedge and stronger leisure preferences incentivize earlier exit from the workforce.

\subsection{Example with GBM Benchmark under Consistent ELS Risk Cost}\label{sec:ex2}
This subsection provides an explicit example in which the investor tracks a GBM benchmark process and faces identical  ELS risk management costs before and after retirement, i.e., $\alpha=\beta$. 
To obtain a closed-form solution in this setting, we consider the constant labor income and leisure levels, i.e., $r_t \equiv r_c$ and $\ell_t \equiv \ell_c$. This special case isolates the effect of the benchmark dynamics from the differential cost structure studied earlier. By setting the pre- and post-retirement shortfall costs equal, we can focus exclusively on how the stochastic benchmark $Z$ influences the investor's retirement decision, optimal portfolio choice, and consumption policy.

We first have the following lemma whose proof is similar to that of Lemma~\ref{lem:ex1}, and hence we omit it.

\begin{lemma}
Consider $\alpha=\beta$, $r_t\equiv r_c$ and $\ell_t\equiv \ell_c$. Then, the solution of VI~\eqref{eq:tildev} is given by $\tilde v(y,z)=\tilde G(y,z)+ P(y)$ for $y\in(0,\alpha]$, where $P(y)$ for $y\in(0,\alpha]$ has the following explicit representation: 
\begin{itemize}
\item[{\rm(i)}] if $\alpha \leq \frac{\ell_c}{r_c}$, then $P(y)\equiv 0$ for all $y\in(0,\alpha]$;
\item [{\rm(ii)}] if $\alpha> \frac{\ell_c}{r_c}$, then we have
\begin{align*}
P(y)=\left(C_1 y+C_2{y}^{-\frac{\rho}{\xi}}-\frac{r_c}{\rho+\xi}y\ln y-\frac{\ell_c}{\rho}\right)\mathds 1_{\{y>y_B\}},\quad \forall y\in(0,\alpha],
\end{align*}
where the constants $C_1=\left(\frac{r_c(1+\ln y_B)}{\rho+\xi}+\frac{\ell_c}{\rho+\xi}y_B^{-1}-\frac{\rho r_c}{(\rho+\xi)^2}\right)$, $C_2=\frac{\ell_c\xi/\rho-r_c\xi/(\rho+\xi)y_B}{\rho+\xi}{y_B}^{\frac{\rho}{\xi}}$ and $y_B\in (0,\frac{\ell_c}{r_c}]$ is the unique solution of the following equation given by
\begin{align*}
\left(\frac{\ell_c/r_c }{y_B}-\frac{ \rho }{\rho+\xi}\right)\left[1- \left(\frac{y_B}{\alpha}\right)^{\frac{\rho}{\xi}+1}\right]+\ln\left(\frac{y_B}{\alpha}\right) = 0.
\end{align*}
Furthermore, the optimal retirement boundary for the original problem is given by
\begin{align*}
x_B(z)=-\tilde G_y(y_B,z)=\frac{(1-p)^2\left(y_B^{-\frac{1}{1-p}}-\alpha^{-\frac{1}{1-p}}\right)}{\rho (1-p)-\xi p}-z\left(1-(y_B/\alpha)^{\frac{\rho}{\xi}-1}\right),~~\forall z>0.    
\end{align*}
\end{itemize} 
\end{lemma}

Figure~\ref{fig:valuezs} plots the dual-space value functions \(\tilde{v}(y,z)\), \(\tilde{G}(y,z)\) and primal-space value functions \(v(x,z)\), \(G(x,z)\) across different tracking benchmark states \(z \in \{0,10,30\}\). 
As the benchmark $z$ rises, the entire value function profile shifts downward in the dual space and upward in the primal space, accompanied by a rightward shift of the endogenous retirement boundary. This monotonicity of the retirement boundary with respect to the benchmark state is formally established in the following Lemma~\ref{lem:Pz}, which also holds for the general case $\alpha\leq\beta$. Intuitively, a higher tracking benchmark state $z$ raises the required wealth level to meet the tracking target, strengthening the incentive for wealth accumulation during working life. Consequently, the investor delays retirement and increases the targeted wealth threshold \(x_B(z)\). At the same time, the widening gap between \(v(x,z)\) and \(G(x,z)\) at higher wealth levels highlights the growing value of the continuation option to remain employed, as meeting the elevated benchmark target becomes more rewarding. Conversely, as wealth approaches \(x_B(z)\), the marginal benefit of continued work diminishes, and the value functions converge toward the retirement-only benchmark.

\begin{lemma}\label{lem:Pz}
Let Assumptions $(\boldsymbol{A}_{\boldsymbol{Z}})$ and    $(\boldsymbol{A}_{\boldsymbol{\rho}})$ hold. Then, if $\ell_c=r_c=0$ or $\ell=r=0$, the boundaries $z\mapsto y_B(z)$ defined in
\eqref{eq:stopingboundary}  and  $z\mapsto x_B(z)$ defined in
\eqref{eq:x_Bz} are non-decreasing on $\R_+$.
\end{lemma}
\begin{figure}[h!]
    \centering
    \includegraphics[width=0.95\linewidth]{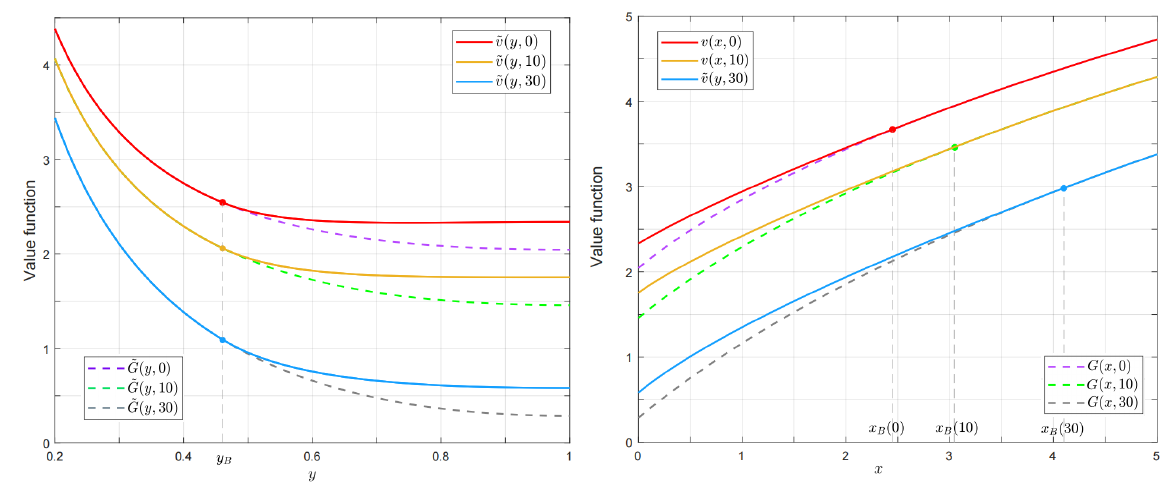}
    \caption{Comparison of value functions $\tilde{v}(y,z)$, $\tilde{G}(y,z)$, ${v}(y,z)$, and ${G}(y,z)$ with respect to      different $z$. The baseline parameters are set as $\rho = 0.7$, $\alpha = \beta = 1$, $p = 0.45$, $\mu_Z = 0.05$, $\sigma_Z = 0.01$, $\sigma = 0.2$, $\gamma = 1$, $\mu = 0.1$, where different color curves represent the results under adjusted $z\in\{0,10,30\}$.}
    \label{fig:valuezs}
\end{figure}

Figure~\ref{fig:thetazs}
 plots the optimal portfolio and consumption policies across different tracking-benchmark states $z$. The left panel illustrates that as $z$ rises, the optimal risky asset holdings \(\theta^*(x,z)\) become more conservative, with the entire portfolio profile shifting downward and the sensitivity to wealth weakening, while the endogenous retirement threshold \(x_B(z)\) shifts rightward. Simultaneously, the right panel illustrates that the optimal consumption \(c^*(x,z)\) is uniformly lower at all wealth levels for a smaller $z$, resulting in a flatter consumption profile over the working period. These patterns arise because a higher benchmark state imposes a more demanding tracking target, raising the wealth required to meet the benchmark at retirement and amplifying the perceived risk of falling short of the target. Consequently, investors adopt a more prudent investment strategy to mitigate downside risk, delay retirement to extend wealth accumulation, and suppress current consumption to strengthen savings buffers.
\begin{figure}[h!]
    \centering
    \includegraphics[width=0.95\linewidth]{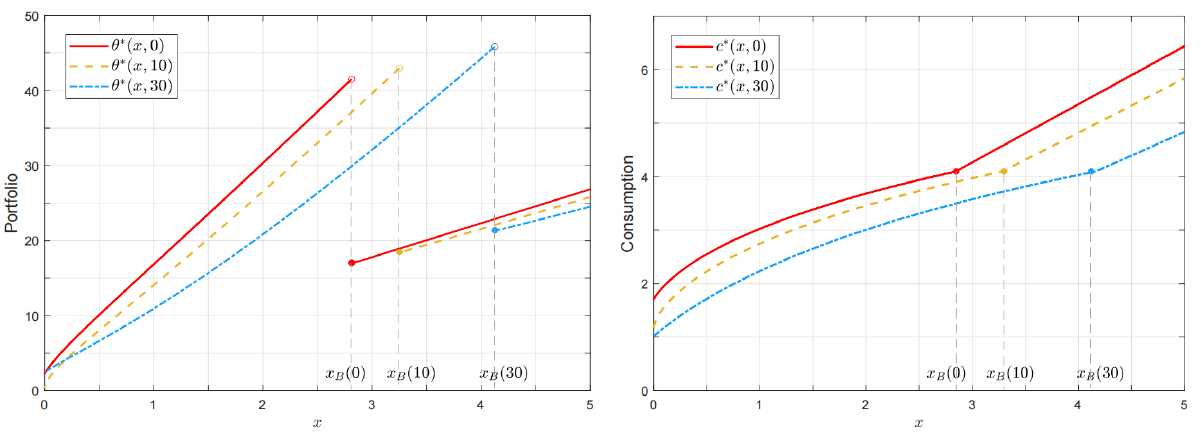}
    \caption{Comparison of   the optimal portfolio and consumption $(\theta^*,c^*)$ with respect to different $z$. The baseline parameters are set as $\rho = 0.7$, $\alpha = \beta = 1$, $p = 0.45$, $\mu_Z = 0.05$, $\sigma_Z = 0.01$, $\sigma = 0.2$, $\gamma = 1$, $\mu = 0.1$, where different color curves represent the results under adjusted $z\in\{0,10,30\}$.}
    \label{fig:thetazs}
\end{figure}

Figure~\ref{fig:ELS_z} illustrates the expected largest shortfall risk as a function of the investor's wealth level $x$, evaluated at three benchmark states $z\in\{0,5,10\}$. When $z=0$, the expected largest shortfall risk increases monotonically with wealth. For higher benchmark states, $z=5$ and $z=10$, the curves display a mild U-shaped pattern: the risk decreases slightly at very low wealth levels and then increases as wealth grows. The initial decline reflects the fact that a marginal increase in wealth provides an immediate buffer against benchmark shortfall when the investor starts from a severely underfunded position. Once wealth becomes sufficiently high, however, the expected largest shortfall risk increases with $x$. This increase is driven by the interaction between retirement timing and benchmark tracking: higher wealth brings the investor closer to the endogenous retirement boundary, after which labor income is lost, while the investor must still manage future benchmark shortfalls.
The figure also shows that the expected largest shortfall risk shifts upward as the benchmark state $z$ increases. This is intuitive because a higher initial benchmark level raises the gap $Z-V$ and therefore increases the potential running maximum shortfall. 
\begin{figure}[h!]
    \centering
    \includegraphics[width=0.55\linewidth]{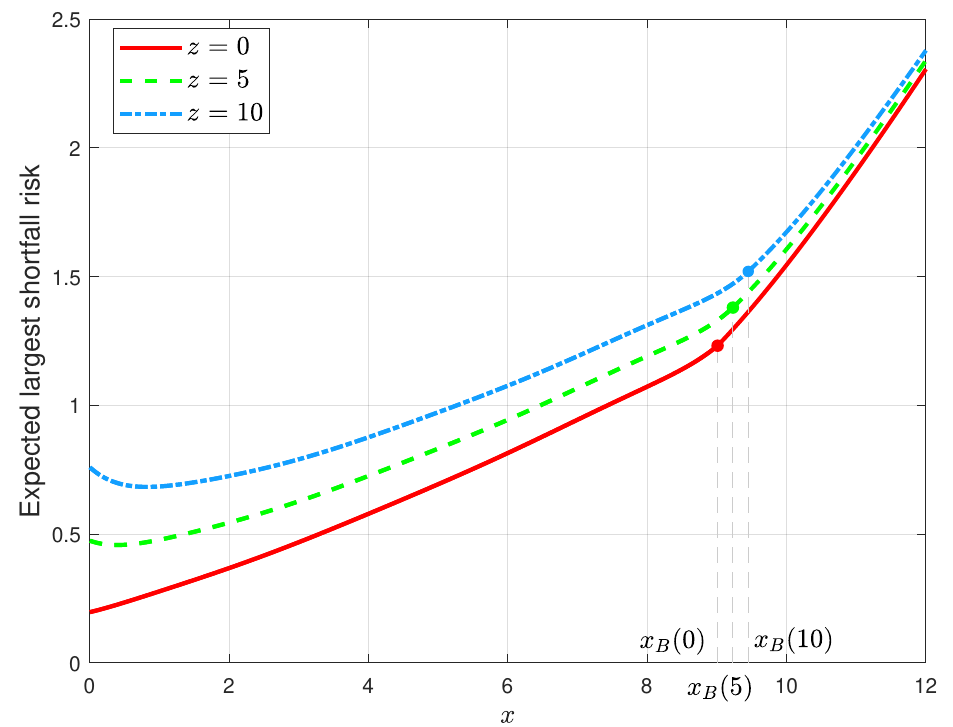}
    \caption{The expected largest shortfall risk for different $z$. The baseline parameters are set as $\rho = 0.7$, $\alpha =1$, $\beta = 1$, $p = 0.7$, $\sigma = 0.1$, $\sigma_Z=0.01$,  $\gamma = 1$, $\mu = 0.1$, $\mu_Z=0.05$, $r_c=2$ and $\ell_c=1.5$.}
    \label{fig:ELS_z}
\end{figure}
\section{Conclusion}\label{sec:conc}

This paper develops a tractable framework for a class of hybrid optimal stopping and control problems under benchmark tracking with relative performance evaluation. By introducing an auxiliary state process and establishing a convex duality approach, the original nonlinear problem with a maximum process is transformed into a two-dimensional optimal stopping problem with state reflection. This transformation enables a clear characterization of the retirement boundary, as well as explicit feedback-form solutions for optimal portfolio, consumption, and retirement decisions. 
Our main results reveal that retirement timing is governed by a wealth- and benchmark-dependent threshold, with optimal strategies exhibiting a two-stage structure across the working and retirement phases. Shortfall management costs, labor income, leisure preference, and benchmark intensity jointly shape investment behavior, consumption, retirement timing, and the expected largest shortfall risk. In particular, shortfall risk displays a U-shaped pattern in wealth, and stricter benchmark requirements induce more conservative behavior, higher wealth targets, and delayed retirement.

This study offers several avenues for future research.  First, allowing for heterogeneous agents and equilibrium interactions would enable the analysis of how benchmark-driven behavior aggregates in the market. Incorporating agents with different risk preferences, income processes, or benchmark sensitivities could reveal potential amplification effects, such as herding or systemic risk arising from relative performance concerns. Second, introducing impulse control features would allow the model to capture discrete, lumpy decisions such as one-time adjustments to the benchmark target, partial withdrawals, or large-scale portfolio rebalancing. Last but not least, a promising avenue is to develop reinforcement learning approaches for investment–consumption problems with endogenous retirement options and benchmark constraints in stochastic environments.  Embedding the current framework into a reinforcement learning setting would allow agents to learn optimal portfolio, consumption, and retirement policies directly from data under evolving market conditions and shifting benchmarks. Such an approach could capture adaptive behavior, handle model uncertainty and provide a flexible tool for solving high-dimensional problems where analytical methods become intractable, thereby enhancing the model’s relevance for real-world financial decision-making.

\bibliographystyle{informs2014}
\bibliography{ref}

@book{shiryaev2008optimal,
  title={Optimal Stopping Rules},
  author={Shiryaev, Albert N},
  year={2008},
  publisher={Springer}
}

@article{BodieMertonSamuelson1992,
  title={Labor supply flexibility and portfolio choice in a life cycle model},
  author={Bodie, Zvi and Merton, Robert C and Samuelson, William F},
  journal={Journal of Economic Dynamics and Control},
  volume={16},
  number={3-4},
  pages={427--449},
  year={1992},
  publisher={Elsevier}
}

@book{karatzas1998methods,
  title={Methods of Mathematical Finance},
  author={Karatzas, Ioannis and Shreve, Steven E and Karatzas, I and Shreve, Steven E},
  volume={39},
  year={1998},
  publisher={Springer}
}

@book{friedman1982variational,
  title={Variational Principles and Free-Boundary Problems},
  author={Friedman, Avner},
  year={1982},
  publisher={John Wiley \& Sons}
}

@article{dybvig2010lifetime,
  title={Lifetime consumption and investment: retirement and constrained borrowing},
  author={Dybvig, Philip H and Liu, Hong},
  journal={Journal of Economic Theory},
  volume={145},
  number={3},
  pages={885--907},
  year={2010},
  publisher={Elsevier}
}

@article{farhi2007saving,
  title={Saving and investing for early retirement: A theoretical analysis},
  author={Farhi, Emmanuel and Panageas, Stavros},
  journal={Journal of Financial Economics},
  volume={83},
  number={1},
  pages={87--121},
  year={2007},
  publisher={Elsevier}
}

@article{yang2018optimal,
  title={Optimal consumption and portfolio selection with early retirement option},
  author={Yang, Zhou and Koo, Hyeng Keun},
  journal={Mathematics of Operations Research},
  volume={43},
  number={4},
  pages={1378--1404},
  year={2018},
  publisher={INFORMS}
}

@article{jang2024optimal,
  title={Optimal investment, heterogeneous consumption, and best time for retirement},
  author={Jang, Hyun Jin and Xu, Zuo Quan and Zheng, Harry},
  journal={Operations Research},
  volume={72},
  number={2},
  pages={832--847},
  year={2024},
  publisher={INFORMS}
}

@article{jin2006disutility,
  title={Disutility, optimal retirement, and portfolio selection},
  author={Choi, Kyoung Jin and Shim, Gyoocheol},
  journal={Mathematical Finance},
  volume={16},
  number={2},
  pages={443--467},
  year={2006},
  publisher={Wiley Online Library}
}

@book{peskir2006optimal,
  title={Optimal {S}topping and {F}ree-{B}oundary {P}roblems},
  author={Peskir, Goran and Shiryaev, Albert},
  year={2006},
  publisher={Springer-Verlag}
}

@book{gilbarg1977elliptic,
  title={Elliptic {P}artial {D}ifferential {E}quations of {S}econd {O}rder},
  author={Gilbarg, David and Trudinger, Neil S},
  volume={224},
  number={2},
  year={1977},
  publisher={Springer}
}

@article{bo2021optimal,
  title={Optimal tracking portfolio with a ratcheting capital benchmark},
  author={Bo, Lijun and Liao, Huafu and Yu, Xiang},
  journal={SIAM Journal on Control and Optimization},
  volume={59},
  number={3},
  pages={2346--2380},
  year={2021},
  publisher={SIAM}
}

@article{browne1999reaching,
  title={Reaching goals by a deadline: Digital options and continuous-time active portfolio management},
  author={Browne, Sid},
  journal={Advances in Applied Probability},
  volume={31},
  number={2},
  pages={551--577},
  year={1999},
  publisher={Cambridge University Press}
}

@article{merton1969lifetime,
  title={Lifetime portfolio selection under uncertainty: The continuous-time case},
  author={Merton, Robert C},
  journal={The Review of Economics and Statistics},
  pages={247--257},
  year={1969},
  publisher={JSTOR}
}

@article{bensoussan1974nonlinear,
  title={Nonlinear variational inequalities and differential games with stopping times},
  author={Bensoussan, Alain and Friedman, Avner},
  journal={Journal of Functional Analysis},
  volume={16},
  number={3},
  pages={305--352},
  year={1974},
  publisher={Elsevier}
}

@article{farkas2021intra,
  title={Intra-Horizon expected shortfall and risk structure in models with jumps},
  author={Farkas, Walter and Mathys, Ludovic and Vasiljevi{\'c}, Nikola},
  journal={Mathematical Finance},
  volume={31},
  number={2},
  pages={772--823},
  year={2021},
  publisher={Wiley Online Library}
}

@article{pham2002minimizing,
  title={Minimizing shortfall risk and applications to finance and insurance problems},
  author={Pham, Huy{\^e}n},
  journal={The Annals of Applied Probability},
  volume={12},
  number={1},
  pages={143--172},
  year={2002},
  publisher={Institute of Mathematical Statistics}
}

@book{fleming2006controlled,
  title={Controlled Markov Processes and Viscosity Solutions},
  author={Fleming, Wendell H and Soner, H Mete},
  year={2006},
  publisher={Springer}
}

@article{friedman1973stochastic,
  title={Stochastic games and variational inequalities},
  author={Friedman, Avner},
  journal={Archive for Rational Mechanics and Analysis},
  volume={51},
  number={5},
  pages={321--346},
  year={1973},
  publisher={Springer}
}

@book{yong1999stochastic,
  title={Stochastic Controls: Hamiltonian Systems and HJB Equations},
  author={Yong, Jiongmin and Zhou, Xun Yu},
  volume={43},
  year={1999},
  publisher={Springer Science \& Business Media}
}

@article{karatzas1986explicit,
  title={Explicit solution of a general consumption/investment problem},
  author={Karatzas, Ioannis and Lehoczky, John P and Sethi, Suresh P and Shreve, Steven E},
  journal={Mathematics of Operations Research},
  volume={11},
  number={2},
  pages={261--294},
  year={1986},
  publisher={INFORMS}
}

@article{karatzas1987optimal,
  title={Optimal portfolio and consumption decisions for a “small investor” on a finite horizon},
  author={Karatzas, Ioannis and Lehoczky, John P and Shreve, Steven E},
  journal={SIAM Journal on Control and Optimization},
  volume={25},
  number={6},
  pages={1557--1586},
  year={1987},
  publisher={SIAM}
}

@article{cox1989optimal,
  title={Optimal consumption and portfolio policies when asset prices follow a diffusion process},
  author={Cox, John C and Huang, Chi F},
  journal={Journal of Economic Theory},
  volume={49},
  number={1},
  pages={33--83},
  year={1989},
  publisher={Elsevier}
}

@article{karatzas2000utility,
  title={Utility maximization with discretionary stopping},
  author={Karatzas, Ioannis and Wang, Hui},
  journal={SIAM Journal on Control and Optimization},
  volume={39},
  number={1},
  pages={306--329},
  year={2000},
  publisher={SIAM}
}

@article{yang2021optimal,
  title={Optimal retirement in a general market environment},
  author={Yang, Zhou and Koo, Hyeng Keun and Shin, Yong Hyun},
  journal={Applied Mathematics \& Optimization},
  volume={84},
  number={1},
  pages={1083--1130},
  year={2021},
  publisher={Springer}
}

@article{park2023robust,
  title={Robust retirement with return ambiguity: Optimal-stopping time in dual space},
  author={Park, Kyunghyun and Wong, Hoi Ying},
  journal={SIAM Journal on Control and Optimization},
  volume={61},
  number={3},
  pages={1009--1037},
  year={2023},
  publisher={SIAM}
}

@article{chen2022optimal,
  title={Optimal retirement under partial information},
  author={Chen, Kexin and Jeon, Junkee and Wong, Hoi Ying},
  journal={Mathematics of Operations Research},
  volume={47},
  number={3},
  pages={1802--1832},
  year={2022},
  publisher={INFORMS}
}

@article{jeon2020optimal,
  title={Optimal retirement and portfolio selection with consumption ratcheting},
  author={Jeon, Junkee and Park, Kyunghyun},
  journal={Mathematics and Financial Economics},
  volume={14},
  number={3},
  pages={353--397},
  year={2020},
  publisher={Springer}
}

@article{bo2026extended,
  title={An extended {M}erton problem with relaxed benchmark tracking},
  author={Bo, Lijun and Huang, Yijie and Yu, Xiang},
  journal={Mathematical Finance},
  volume={36},
  number={2},
  pages={422--448},
  year={2026},
  publisher={Wiley Online Library}
}

@article{strub2018optimal,
  title={Optimal construction and rebalancing of index-tracking portfolios},
  author={Strub, Oliver and Baumann, Philipp},
  journal={European Journal of Operational Research},
  volume={264},
  number={1},
  pages={370--387},
  year={2018},
  publisher={Elsevier}
}

@article{ni2022optimal,
  title={Optimal asset allocation for outperforming a stochastic benchmark target},
  author={Ni, Chendi and Li, Yuying and Forsyth, Peter and Carroll, Ray},
  journal={Quantitative Finance},
  volume={22},
  number={9},
  pages={1595--1626},
  year={2022},
  publisher={Taylor \& Francis}
}

@article{yao2006tracking,
  title={Tracking a financial benchmark using a few assets},
  author={Yao, David D and Zhang, Shuzhong and Zhou, Xun Yu},
  journal={Operations Research},
  volume={54},
  number={2},
  pages={232--246},
  year={2006},
  publisher={INFORMS}
}

@article{dybvig2011verification,
  title={Verification theorems for models of optimal consumption and investment with retirement and constrained borrowing},
  author={Dybvig, Philip H and Liu, Hong},
  journal={Mathematics of Operations Research},
  volume={36},
  number={4},
  pages={620--635},
  year={2011},
  publisher={INFORMS}
}

@article{choi2008optimal,
  title={Optimal portfolio, consumption-leisure and retirement choice problem with {CES} utility},
  author={Choi, Kyoung Jin and Shim, Gyoocheol and Shin, Yong Hyun},
  journal={Mathematical Finance},
  volume={18},
  number={3},
  pages={445--472},
  year={2008},
  publisher={Wiley Online Library}
}

@article{gaivoronski2005optimal,
  title={Optimal portfolio selection and dynamic benchmark tracking},
  author={Gaivoronski, Alexei A and Krylov, Sergiy and Van der Wijst, Nico},
  journal={European Journal of Operational Research},
  volume={163},
  number={1},
  pages={115--131},
  year={2005},
  publisher={Elsevier}
}

@article{bo2026optimal,
  title={Optimal consumption under relaxed benchmark tracking and consumption drawdown constraint},
  author={Bo, Lijun and Huang, Yijie and Yan, Kaixin and Yu, Xiang},
  journal={SIAM Journal on Financial Mathematics},
  volume={17},
  number={1},
  pages={78--117},
  year={2026},
  publisher={SIAM}
}

@article{bo2025stochastic,
  title={Stochastic control problems with state reflections arising from relaxed benchmark tracking},
  author={Bo, Lijun and Huang, Yijie and Yu, Xiang},
  journal={Mathematics of Operations Research},
  volume={50},
  number={4},
  pages={2526--2551},
  year={2025},
  publisher={INFORMS}
}

@article{merton1971optimum,
  title={Optimum consumption and portfolio rules in a continuous-time model},
  author={Merton, Robert C},
  journal={Journal of Economic Theory},
  volume={3},
  number={4},
  pages={373--413},
  year={1971},
  publisher={Elsevier}
}

@article{browne2000risk,
  title={Risk-constrained dynamic active portfolio management},
  author={Browne, Sid},
  journal={Management Science},
  volume={46},
  number={9},
  pages={1188--1199},
  year={2000},
  publisher={INFORMS}
}

@article{browne1999beating,
  title={Beating a moving target: Optimal portfolio strategies for outperforming a stochastic benchmark},
  author={Browne, Sid},
  journal={Finance and Stochastics},
  volume={3},
  number={3},
  pages={275--294},
  year={1999},
  publisher={Springer}
}

@article{de2015nonconvex,
  title={A nonconvex singular stochastic control problem and its related optimal stopping boundaries},
  author={De Angelis, Tiziano and Ferrari, Giorgio and Moriarty, John},
  journal={SIAM Journal on Control and Optimization},
  volume={53},
  number={3},
  pages={1199--1223},
  year={2015},
  publisher={SIAM}
}



\newpage 
\vspace{0.3cm}
\pagenumbering{arabic}
\section*{\center Electronic Companion (EC) To
  ``Optimal Consumption and\\  Retirement Time under Shortfall Risk Measure'' \\ {\normalsize	  Proofs of Auxiliary Results}}
\setcounter{equation}{0}
\setcounter{page}{1}\def\thepage{ec\arabic{page}}%
\renewcommand{\theequation}{EC.\arabic{equation}}
In this electronic companion, we provide proofs of the auxiliary results required for proving the main results in the paper’s main body.

\noindent{\bf Proof of Lemma~\ref{lem:XY}.}\quad
Let $Y_t:=Y_t^\alpha \mathds 1_{\{t<\tau\}}+Y_t^{\tau,\beta} \mathds 1_{\{t\geq \tau\}} $ and $L_t^Y:=\int_0^t \mathds 1_{\{s<\tau\}} d L_s^{\alpha}+ \int_0^t\mathds 1_{\{s\geq \tau\}} d L_s^{\beta,\tau}$ for $t\geq0$. Then, integration by parts yields that
\begin{align*}
    dY_t&= \mathds 1_{\{t<\tau\}} dY_t^\alpha+ \mathds 1_{\{t\geq \tau\}}d Y_t^{\tau,\beta}+(Y_{\tau}^{\tau,\beta}-Y_{\tau }^\alpha)\mathds 1_{\{\tau\in dt\}}=\mathds 1_{\{t<\tau\}} dY_t^\alpha+ \mathds 1_{\{t\geq \tau\}}d Y_t^{\tau,\beta}\notag\\
    &=\rho Y_t dt-\mu^\top (\sigma\sigma^\top)^{-1}\sigma Y_t dW_t-dL_t^Y.
\end{align*}
It follows from It\^o's rule that
\begin{align}\label{eq:tauif}
d\left(e^{-\rho t}X_tY_t\right)
&=e^{-\rho t}\phi_tY_t\,dW_t
-e^{-\rho t}X_t\,dL_t^Y
+e^{-\rho t}Y_t\,dL_t^X  \notag\\
&\quad
-e^{-\rho t}
\left(c_t+\mu_XZ_t-(rZ_t+r_c)\mathds 1_{\{t<\tau\}}\right)Y_tdt
\end{align}
with $\phi_t:=\theta_t^{\top}\sigma -\sigma_Z Z_t  {\gamma}^{\top}-\mu^{\top}(\sigma \sigma^\top)^{-1} \sigma X_t$ for $t\geq0$. 
Integrating with respect to $t$ on both sides of \eqref{eq:tauif} from $0$ to $T>0$, we arrive at
\begin{align}\label{eq:XYtau}
& e^{-\rho T}X_TY_T+\int_0^T e^{-\rho t}\left\{(c_t + \mu_X Z_t\mathds 1_{\{t<\tau\}})Y_t d t+ X_tdL_t^Y\right\}\notag\\
&\qquad= xy+\int_0^T e^{-\rho t} (rZ_t+r_c)\mathds 1_{\{t<\tau\}}dt+\int_0^T e^{-\rho t}  Y_t d L_t^X +\int_0^T e^{-\rho t} \phi_t Y_t dW_t:=M_T.
\end{align}
Since $\mu_X-r\geq 0$ and both $X,Y$ are nonnegative, we have from \eqref{eq:XYtau} that the process $M_T \geq 0$, a.s..  Define the stopping time $\tau_n:=\inf\{t\geq0;~\int_0^t|\phi_s Y_s|^2ds \geq n  \}$. Consequently,  $\tau_n\to\infty$ as $n\to\infty$. Recall that $0<Y_t^{\alpha}\leq \alpha$ and $0<Y_t^{\tau,\beta}\leq \beta$, $\Px$-a.s.. It follows from Fatou's lemma, the optional sampling theorem, and the monotone convergence theorem (MCT) that
\begin{align}\label{eq:EMT}
\mathbb{E}\left[M_{T }\right]&=\mathbb{E}\left[\liminf_{n \rightarrow \infty} M_{T\wedge\tau_n}\right] \leq \liminf_{n \rightarrow \infty} \mathbb{E}\left[M_{T \wedge\tau_n}\right]\nonumber\\
&=x y+\liminf_{n \rightarrow \infty} \mathbb{E}\left[\int_0^{T\wedge\tau_n} e^{-\rho t} (rZ_t+r_c)\mathds 1_{\{t<\tau\}}dt+\int_0^{T\wedge\tau_n} e^{-\rho t} Y_t d L_t^X+\int_0^{T \wedge\tau_n} e^{-\rho t} \phi_t Y_t dW_t\right]\nonumber \\
& =x y+\liminf_{n \rightarrow \infty} \mathbb{E}\left[\int_0^{T\wedge\tau_n} e^{-\rho t} (rZ_t+r_c)\mathds 1_{\{t<\tau\}}dt+\int_0^{T\wedge\tau_n} e^{-\rho t} Y_t d L_t^X\right]\notag\\
&=x y+\mathbb{E}\left[\int_0^{T} e^{-\rho t} (rZ_t+r_c)\mathds 1_{\{t<\tau\}}dt+\int_0^{T} e^{-\rho t} Y_t d L_t^X\right].
\end{align}
Letting $T\to \infty$ and using MCT again, we derive the desired budget constraint \eqref{eq:contrain}. \hfill$\Box$\\

\noindent{\bf Proof of Lemma~\ref{lem:growth}.}\quad
Noting that $Y^{\alpha}=(Y^{\alpha}_t)_{t\geq0}$ takes values on $(0,\alpha]$. Then, we have 
\begin{align*}
&\Ex\left[\int_0^{\infty} e^{-\rho t} \left|\tilde U(Y_t^\alpha)+(r-\mu_X) Z_t Y_t^\alpha-\ell Z_t+r_cY_t^\alpha-\ell_c\right|dt\right]\nonumber\\
&\quad\leq \Ex\left[\int_0^{\infty} e^{-\rho t} \left(\left|\tilde U(Y_t^\alpha)\right|+(r+\mu_X) Z_t Y_t^\alpha+\ell Z_t+r_cY_t^\alpha+\ell_c\right)dt\right]\nonumber\\
&\quad\leq \Ex\left[\int_0^{\infty} e^{-\rho t} \left(\left|\tilde U(Y_t^\alpha)\right|+ \alpha(r+\mu_X) Z_t+\ell Z_t+\alpha r_c+\ell_c\right)dt\right]\nonumber\\
&\quad=\frac{1}{|p|}\left(\dfrac{(1 - p)^3}{\rho(1 - p) - \xi p} y^{-\frac{p}{1 - p}} + \dfrac{p(1 - p)^2}{\rho(1 - p) - \xi p} \alpha^{-\frac{1}{1 - p}} y\right)+\frac{\alpha(r+\mu_X)+\ell}{\rho -\mu_Z}z+\frac{\alpha r_c+\ell_c}{\rho}\nonumber\\
&\quad<+\infty.
\end{align*}

We next show $\Ex[\sup_{t\geq 0}|
e^{-\rho t}\tilde G(Y_t^\alpha,Z_t)|]<+\infty$. Since $y \mapsto \tilde G(y,z)$ and $z\mapsto \tilde G(y,z)$ are both decreasing, $G(\alpha,z)\leq G(y,z)\leq G(y,0)$ for all $(y,z)\in(0,\alpha]\times\R_+$. Note that
\begin{align}\label{eq:Z1}
&\Ex\left[\sup_{t\geq 0}\left|
e^{-\rho t}\tilde G(\alpha,Z_t)\right|\right]\nonumber\\
&\quad\leq \dfrac{(1 - p)^3}{p(\rho(1 - p) - \xi p)} \alpha^{-\frac{p}{1 - p}} + \dfrac{(1 - p)^2}{\rho(1 - p) - \xi p} \beta^{-\frac{1}{1 - p}} \alpha +  \left| \alpha - \dfrac{\beta^{-(\kappa - 1)}}{\kappa} \alpha^\kappa \right|\Ex\left[\sup_{t\geq 0}
\left(e^{-\rho t}Z_t\right)\right].
\end{align}
By using Burkholder-Davis-Gundy (BDG) inequality, it holds that
\begin{align}\label{eq:Z2}
&\Ex\left[\sup_{t\geq 0}
\left(e^{-\rho t}Z_t\right)\right]=z+\Ex\left[\sup_{t\geq 0}\left(\int_0^t e^{-\rho s}\mu_Z Z_sds +\int_0^t e^{-\rho s}\sigma_Z Z_sdW^{\gamma}_s\right)\right]\nonumber\\
&\quad\leq z+\Ex\left[\int_0^{\infty} e^{-\rho s}|\mu_Z| Z_sds \right]+3\Ex\left[\left(\int_0^{\infty} e^{-2\rho s}\sigma_Z^2 Z_s^2ds\right)^{\frac{1}{2}}\right]\nonumber\\
&\quad\leq z+\frac{|\mu_Z|}{\rho-\mu_Z} +3|\sigma_Z|\left(\Ex\left[\int_0^{\infty} e^{-2\rho s} Z_s^2ds\right]\right)^{\frac{1}{2}}\leq z+\frac{|\mu_Z|}{\rho-\mu_Z} +3|\sigma_Z|\left(\frac{1}{\rho-2\mu_Z-\sigma_Z^2}\right)^{\frac{1}{2}}.
\end{align}
It follows from \eqref{eq:Z1} and \eqref{eq:Z2} that $\Ex[\sup_{t\geq 0}|e^{-\rho t}\tilde G(\alpha,Z_t)|]<+\infty$.

On the other hand, we have from Lemma \ref{lem:struct} that 
\begin{align}\label{eq:Y0}
&\Ex\left[\sup_{t\geq 0}\left|
e^{-\rho t}\tilde G(Y_t^{\alpha},0)\right|\right]\leq \dfrac{(1 - p)^3}{|p|(\rho(1 - p) - \xi p)}\Ex\left[\sup_{t\geq 0} e^{-\rho t}\left(Y_t^{\alpha}\right)^{\frac{p}{p-1}}\right] + \dfrac{(1 - p)^2}{\rho(1 - p) - \xi p} \beta^{-\frac{1}{1 - p}} \alpha.
\end{align}
By applying It\^o's formula, we have $e^{-\rho t} Y_t^{\alpha} = y  -\int_0^t  e^{-\rho s}d L_s^{\alpha}- {\mu^{ \top} (\sigma \sigma^\top)^{-1}\sigma } \int_0^t e^{-\rho s} Y_s^{\alpha}dW_s$ and 
\begin{align}\label{eq:Y2}
e^{-\rho t}\left(Y_t^{\alpha}\right)^{\frac{p}{p-1}}&=y^{\frac{p}{p-1}}+
\int_0^t \frac{p \xi-\rho (1-p)}{(1-p)^2}e^{-\rho s}\left(Y_s^{\alpha}\right)^{\frac{p}{p-1}} ds -\frac{p \alpha^\frac{1}{p-1} }{p-1}\int_0^t  e^{-\rho s}d L_s^{\alpha}\nonumber\\
&\quad- \frac{p\mu^{ \top} (\sigma \sigma^\top)^{-1}\sigma }{p-1} \int_0^t e^{-\rho s} (Y_s^{\alpha})^{\frac{p}{p-1}}dW_s.
\end{align}
Taking the expectation on both sides of the above equalities, we respectively have $\Ex[\int_0^t  e^{-\rho s}d L_s^{\alpha}]=y-\Ex\left[e^{-\rho t}Y_t^{\alpha}\right]\leq y$ and
\begin{align}\label{eq:Y4}
\Ex\left[e^{-\rho t}\left(Y_t^{\alpha}\right)^{\frac{p}{p-1}}\right]\leq y^{\frac{p}{p-1}} +\frac{p \alpha^\frac{1}{p-1} }{1-p}\Ex\left[\int_0^t  e^{-\rho s}d L_s^{\alpha}\right]
&\leq y^{\frac{p}{p-1}}+\frac{p \alpha^\frac{1}{p-1} }{1-p}y.
\end{align}
It follows from \eqref{eq:Y2}, \eqref{eq:Y4} and BDG inequality that
\begin{align}\label{eq:Y5}
&\Ex\left[\sup_{t\geq 0} e^{-\rho t}\left(Y_t^{\alpha}\right)^{\frac{p}{p-1}}\right]\leq y^{\frac{p}{p-1}}+\frac{p \alpha^\frac{1}{p-1} }{1-p}\Ex\left[\int_0^{\infty}  e^{-\rho s}d L_s^{\alpha}\right]+ \frac{3|p|\mu^{ \top} (\sigma \sigma^\top)^{-1}\sigma }{1-p} \Ex\left[\left( \int_0^{\infty} e^{-2\rho s} (Y_s^{\alpha})^{\frac{2p}{p-1}}ds\right)^{\frac{1}{2}}\right]\nonumber\\
&\qquad\leq y^{\frac{p}{p-1}}+\frac{p \alpha^\frac{1}{p-1} }{1-p}y+ \frac{3|p|\mu^{ \top} (\sigma \sigma^\top)^{-1}\sigma }{1-p} \left(\Ex\left[ \int_0^{\infty} e^{-2\rho s} (Y_s^{\alpha})^{\frac{2p}{p-1}}ds\right]\right)^{\frac{1}{2}}<\infty.
\end{align}
We obtain from \eqref{eq:Y0} and \eqref{eq:Y5} that $\Ex[\sup_{t\geq 0}|e^{-\rho t}\tilde G(Y_t,0)|]<+\infty$. Thus, we complete the proof of the lemma.
\hfill$\Box$\\

\noindent{\bf Proof of Lemma~\ref{lem:struct}.}\quad
For the claim (i), let $Y^{\alpha,y}=(Y_t^{\alpha,y})_{t\geq0}$ and $Z^z=(Z_t^{z})_{t\geq0}$ be respectively the RDP given by \eqref{eq:Y_t} and the benchmark process given by \eqref{eq:Zt} with respective initial values $Y_0^{\alpha,y}=y$ and $Z_0^{z}=z$. Let $\tau^*(y,z)\in\mathbb{T}$ be the optimal stopping time of problem \eqref{eq:tildev}. Then, by using \eqref{eq:tildev}, it follows that, for any $0\leq z_1<z_2$,
{\small\begin{align*}
& \tilde{v}(y, z_1)-\tilde{v}(y, z_2)\\
&\geq \mathbb{E}\left[\int_0^{\tau^*(y,z_2)} e^{-\rho s} \left(\tilde U(Y_s^{\alpha,y})+(r-\mu_X) Z_s^{z_1}Y_s^{\alpha,y}-\ell  Z_s^{z_1}+r_c Y_s^{\alpha,y}-\ell_c\right)ds+e^{-\rho \tau^*(y,z_2)}\tilde G(Y_{\tau^*(y,z_2)}^{\alpha,y},Z_{\tau^*(y,z_2)}^{z_1})\right]\nonumber\\
&\quad-\mathbb{E}\left[\int_0^{\tau^*(y,z_2)} e^{-\rho s} \left(\tilde U(Y_s^{\alpha,y})+(r-\mu_X) Z_s^{z_2}Y_s^{\alpha,y}-\ell  Z_s^{z_2}+r_c Y_s^{\alpha,y}-\ell_c
\right)ds+e^{-\rho \tau^*(y,z_2)}\tilde G(Y_{\tau^*(y,z_2)}^{\alpha,y},Z_{\tau^*(y,z_2)}^{z_2})\right]\nonumber\\
&=\Ex\left[
\int_0^{\tau^*(y,z_2)} e^{-\rho s }(r-\mu_X)( Z_s^{z_1} - Z_s^{z_2})Y_s^{\alpha,y} -\ell (Z_s^{z_1}-Z_s^{z_2})ds\right]\\
&\qquad+\Ex\left[e^{-\rho \tau^*(y,z_2)}\left(
\tilde G(Y_{\tau^*(y,z_2)},Z_{\tau^*(y,z_2)}^{z_1})-\tilde G(Y_{\tau^*(y,z_2)},Z_{\tau^*(y,z_2)}^{z_2})\right)\right]\geq 0.
\end{align*}}This concludes that $z\mapsto\tilde{v}(y,z)$ is decreasing for fixed $y\in(0,\alpha]$.

For the claim (ii), we first show $y \mapsto \tilde{v}(y,z)$ is non-increasing. Define $K_t=-\ln (Y_t^{\alpha}/\alpha)$ for $t>0$ with $K_0=k:=-\ln (y/\alpha)$ and $Y_0^{\alpha}=y\in(0,\alpha]$. By using It\^o's rule together with \eqref{eq:Y_t}, we obtain
\begin{align}\label{eq:Kt}
 K_t &= k(y) + \int_0^t \left(-\rho ds + \mu^{\top}(\sigma\sigma^{\top})^{-1}\sigma dW_s + \frac{dL_s^{\alpha}}{Y_s^{\alpha}}\right) + \frac{1}{2}\int_0^t\mu^{\top}\left(\sigma \sigma^{\top}\right)^{-1} \mu ds.   
\end{align}
Define the process $L_t^{K}:=\int_0^t\frac{dL_s^{\alpha}}{Y_s^{\alpha}}$ for $t\geq0$. Then, $t\to L_t^{K}$ is continuous and non-decreasing. Moreover, it holds that, a.s.
\begin{align*}
\int_{0}^t \mathds{1}_{\{K_s=0\}} dL_s^{K} =   \int_{0}^t \mathds{1}_{\{\ln\frac{Y_s^{\alpha}}{\alpha}=0\}} \frac{dL_s^{\alpha}}{Y_s^{\alpha}}=\int_{0}^t \mathds{1}_{\{Y_s^{\alpha}=\alpha\}} \frac{dL_s^{\alpha}}{Y_s^{\alpha}}=\frac{1}{\alpha}\int_{0}^t \mathds{1}_{\{Y_s^{\alpha}=\alpha\}} dL_s^{\alpha}=\frac{L_t^{\alpha}}{\alpha}=L_t^K,~~\forall t\geq0.
\end{align*}
This yields that $t\to L_t^{K}$ increases on the time set $\{t\geq0;~K_t=0\}$ only. Thus, it follows from \eqref{eq:Kt} that $K=(K_t)_{t\geq0}$ satisfies the following drifted Brownian motion with reflection at the boundary zero: 
\begin{align}\label{eq:Kt2}
 K_t &= k + (\xi-\rho)t +\mu^{\top}(\sigma\sigma^{\top})^{-1}\sigma W_t + L_t^K\geq0,\quad \forall t\geq0,
\end{align}
where, we recall that $\xi=\mu^{\top}(\sigma \sigma^{\top})^{-1} \mu / 2$ is defined in Lemma~\ref{lem:struct}.
The solution representation of ``the Skorokhod problem" yields that, when $Y_0^{\alpha}=y\in(0,\alpha]$,
\begin{align}\label{eq:LtR}
L_t^{K} =(f_t-k)^+,\quad f_t:=\sup_{s\leq t} \left((\rho-\xi)s -\mu^{\top}(\sigma\sigma^{\top})^{-1}\sigma W_s \right).
\end{align} 
Consequently, it holds that, for $t\geq0$,
\begin{align}\label{eq:Yrepresentation}
Y_t^\alpha=\alpha e^{-K_t}=\alpha e^{-k-(f_t-k)^{+}}\eta_t,\quad \eta_t:=e^{(\rho-\xi)t-\mu^\top (\sigma \sigma^\top)^{-1}\sigma W_t}.    
\end{align}
For any $y_1,y_2\in (0,\alpha]$ with  $y_1> y_2$, set $k_i:=\ln(\alpha/y_i)$ (equivalently $k_1<k_2$) and $L_t^{K^i}=(f_t-k_i)^+$ for $t\geq0$. In view of \eqref{eq:LtR}, we have
\begin{align*}
    k_1+L_t^{K^1}-k_2-L_t^{K^2}=\begin{cases}
      \displaystyle  \qquad 0,  &  \text{if } f_t\geq k_2,\\
      \displaystyle f_t-k_2<0, & \text{if } k_1<f_t< k_2,\\
      \displaystyle k_1-k_2<0, & \text{if } f_t\leq k_1.
    \end{cases}
\end{align*}
This yields from \eqref{eq:Yrepresentation} that $Y_t^{\alpha,y_1} :=\alpha e^{-k_1-L_t^{K^{1}}}\eta_t\geq  Y_t^{\alpha,y_2}:=\alpha e^{-k_2-L_t^{K^{2}}}\eta_t$ for all $t\geq0$. In other words, $y\to Y_t^{\alpha}$ is non-decreasing. On the other hand, since $y\to \tilde{U}(y)$ is non-increasing, we have $y\to\tilde{U}(Y_s^{\alpha,y})$ is non-increasing. By the assumption that $r-\mu_X(\cdot)\leq0$, one has $y\to (r-\mu_X)Z_sY_s^{\alpha,y}$ is non-increasing.  Together with the fact that $y\to\tilde{G}(Y_{\tau}^{\alpha,y},Z_{\tau})$ is non-increasing, it follows from \eqref{eq:tildev} that $y\to\tilde{v}(y,z)$ is non-increasing. Next, we show $y \mapsto \tilde{v}(y,z)$ is strictly convex. For $y_1,y_2\in(0,\alpha]$ with $y_1<y_2$ and $\lambda_1,\lambda_2>0$ with $\lambda_1+\lambda_2=1$, set $k_i:=\ln(\alpha/y_i)>0$ (equivalently $k_1>k_2$) and $\hat k:=\ln(\frac{\alpha }{\lambda_1 y_1+\lambda_2 y_2})$. Accordingly, $k_2<\hat k<k_1$. Thus, we have
\begin{align}\label{eq:invtildev}
& \tilde v (\lambda_1 y_1+\lambda_2 y_2,z) - \lambda_1\tilde{v}(y_1,z)-\lambda_2\tilde{v}(y_2,z)\nonumber\\
&\quad\leq \sup_{\tau\in \mathbb T
    }\Ex\bigg[\int_0^\tau e^{-\rho s} \big\{\tilde U(\alpha e^{-\hat k-(f_s-\hat k)^+} \eta_s)-\lambda_1\tilde U(\alpha e^{-k_1-(f_s-k_1)^+} \eta_s)-\lambda_2\tilde U(\alpha e^{-k_2-(f_s-k_2)^+} \eta_s)\nonumber\\
    &\qquad\qquad 
    +(\left(r-\mu_X\right)Z_s+r_c) \alpha \big(e^{-\hat k-(f_s-\hat k)^+}-\lambda_1e^{-k_1-(f_s-k_1)^+}-\lambda_2e^{-k_2-(f_s-k_2)^+}\big)\eta_s \big\}ds\\
    &\quad+e^{-\rho \tau}\big(\tilde G(\alpha e^{-\hat k-(f_\tau-\hat k)^+}\eta_\tau,Z_\tau)-\lambda_1\tilde G(\alpha e^{-k_1-(f_\tau-k_1)^+}\eta_\tau,Z_\tau)-\lambda_2\tilde G(\alpha e^{-k_2-(f_\tau-k_2)^+}\eta_\tau,Z_\tau)\big)\bigg].\nonumber
\end{align}
Denote by
\begin{align*}
\delta_s^{\tilde U} :=\tilde U\left(\alpha e^{-\hat k-(f_s-\hat k)^+} \eta_s\right)-\lambda_1 \tilde U\left(\alpha e^{-k_1-(f_s-k_1)^+} \eta_s\right)-\lambda_2 \tilde U\left(\alpha e^{-k_2-(f_s-k_2)^+} \eta_s \right).
\end{align*}
Note that ${\tilde U}'(\cdot)<0$ and ${\tilde U}''(\cdot)>0$. Then, for $s\geq0$, one has
 \begin{itemize}
    \item  if $f_s\leq k_2$, then $\delta_s^{\tilde U} = \tilde U\left((\lambda_1 y_1+\lambda_2 y_2) \eta_s\right)-\lambda_1 \tilde U\left(y_1 \eta_s\right)-\lambda_2 \tilde U\left( y_2 \eta_s \right)<0$;
    \item if $k_2<f_s\leq \hat k$, then $\delta_s^{\tilde U}= \tilde U((\lambda_1 y_1+\lambda_2 y_2) \eta_s)-\lambda_1 \tilde U(y_1 \eta_s)-\lambda_2 \tilde U(\alpha e^{-f_s } \eta_s)<\tilde U\left((\lambda_1 y_1+\lambda_2 y_2) \eta_s\right)-\lambda_1 \tilde U\left(y_1 \eta_s\right)-\lambda_2 \tilde U\left(y_2 \eta_s\right)<0$.
   
     \item   if $\hat k<f_s< k_1$, then $ \delta_s^{\tilde U} = \tilde U(\alpha e^{-f_s } \eta_s)-\lambda_1 \tilde U(y_1 \eta_s)-\lambda_2 \tilde U(\alpha e^{-f_s } \eta_s){=}\lambda_1( \tilde U(\alpha e^{-f_s } \eta_s)-  \tilde U(y_1 \eta_s))<0$.
    \item    if $ f_s\geq  k_1$, then $\delta_s^{\tilde U} = \tilde U(\alpha e^{-f_s } \eta_s)-\lambda_1 \tilde U(\alpha e^{-f_s } \eta_s)-\lambda_2 \tilde U(\alpha e^{-f_s } \eta_s)=0$.
\end{itemize}
Similarly, it is not difficult to check that, for $s\geq0$, 
\begin{align*}
   \delta_s^r:=((r-\mu_X) Z_s+r_c) \alpha\left(
  e^{-\hat k-(f_s-\hat k)^+}-\lambda_1  e^{-k_1-(f_s-k_1)^+} -\lambda_2 \alpha e^{-k_2-(f_s-k_2)^+} 
\right) \eta_s\leq0.
\end{align*}
On the other hand, since $y\to \tilde G(y,z)$ is strictly convex, it holds that
\begin{align*}
  \delta_\tau^{\tilde G}:=\tilde G\left(\alpha e^{-\hat k-(f_\tau-\hat k)^+}\eta_\tau,Z_\tau\right) 
-\lambda_1\tilde G(e^{-k_1-(f_s-k_1)^+}\eta_\tau,Z_\tau)-\lambda_2\tilde G(e^{-k_2-(f_s-k_2)^+}\eta_\tau,Z_\tau)
<0. 
\end{align*}
Combining the above estimates, it follows from \eqref{eq:invtildev} that $\tilde v (\lambda_1 y_1+\lambda_2 y_2,z) - \lambda_1\tilde{v}(y_1,z)-\lambda_2\tilde{v}(y_2,z)<0$. This shows that $y\to\tilde{v}(y,z)$ is strictly convex. Since $y\to\tilde{v}(y,z)$ is also decreasing, we hence have $y \mapsto \tilde{v}(y, z)$ is also strictly decreasing.
\hfill$\Box$\\

To prove Lemma \ref{lem:continuous}, we should first prove  Lemma \ref{prop:tvy}.\\
\noindent{\bf Proof of Lemma \ref{prop:tvy}.}\quad
For any $\epsilon>0$, set $k:=\ln(\frac{\alpha}{y})$, $u(k,z):=\tilde{v}(\alpha e^{-k},z)$ and $\tau^*_\epsilon:=\tau^*(y+\epsilon,z)=\tau^*(\alpha e^{-k}+\epsilon,z)$. Following the proof of Lemma~\ref{lem:struct}-(i), it holds that
\begin{align*}
e^{-{k}-L_s^{K^{k}}}-e^{-{(k+\epsilon)}-L_s^{K^{k+\epsilon}}}=\begin{cases}
\displaystyle~~~~~~~~~~~~~~~~~~~ 0,  &  \text{if } f_s\geq k+\epsilon,\\
\displaystyle e^{-f_s}-e^{-{(k+\epsilon)}}<e^{-k}-e^{-{(k+\epsilon)}}, & \text{if } k<f_s< k+\epsilon,\\
\displaystyle ~~~~~~~~~e^{-k}-e^{-{(k+\epsilon)}}, & \text{if } f_s\leq k.
    \end{cases}
\end{align*}
Then, using the fact that $\tilde U'(\cdot)<0$ and $\tilde U''(\cdot)>0$ satisfied by the utility function $U(\cdot)$ in \eqref{eq:utility}, we have from the mean value theorem that
\begin{align*}
u(k+\epsilon,z)-u(k,z)&\leq \Ex\Bigg[\int_0^{\tau_\epsilon^*} e^{-\rho s} \bigg(\tilde U\left(\alpha e^{-{(k+\epsilon)}-L_s^{K^{k+\epsilon}}} \eta_s\right)-\tilde U\left(\alpha e^{-{k}-L_s^{K^{k}}} \eta_s\right)\notag\\
&\qquad\quad+((r-\mu_X) Z_s^z+r_c) \alpha \left(e^{-{(k+\epsilon)}-L_s^{K^{k+\epsilon}}}-e^{-{k}-L_s^{K^{k}}}\right)\eta_s  \bigg)ds\notag\\
&\qquad\quad+e^{-\rho{\tau_\epsilon^*}}\left(\tilde G\left(\alpha e^{-{(k+\epsilon)}-L_\tau^{K^{k+\epsilon}}}\eta_\tau,Z_{\tau}^z\right)-\tilde G(\alpha e^{-{k}-L_{\tau_\epsilon^*}^{K^{k}}}\eta_{\tau_\epsilon^*},Z^z_{\tau_\epsilon^*})\right)\Bigg]\\
&\leq \Ex\Bigg[\int_0^{\tau_\epsilon^* \wedge \tau_f^{k+\epsilon}} e^{-\rho s} \left(\tilde U'(\alpha e^{-({k+\epsilon})-L_s^{K^{k+\epsilon}}} \eta_s)+(r-\mu_X) Z_s^z+r_c\right) \left(e^{-{(k+\epsilon)}}-e^{-{k}}\right)\alpha \eta_s ds\notag\\
&\qquad\quad+e^{-\rho {\tau_\epsilon^*}}\tilde G_y\left(\alpha e^{-({k+\epsilon})-L_{\tau_\epsilon^*}^{K^{k+\epsilon}}}\eta_{\tau_\epsilon^*},Z^z_{\tau_\epsilon^*}\right)\left(e^{-{(k+\epsilon)}}-e^{-{k}}\right)\alpha \eta_{\tau_\epsilon^*}\mathds 1_{\left\{\tau_\epsilon^*<\tau_f^{k+\epsilon}\right\}}\Bigg]\notag\\
 &\leq \epsilon \Ex\Bigg[\int_0^{\tau_\epsilon^* \wedge \tau_f^{k+\epsilon}} e^{-\rho s} \left(\tilde U'\left(\alpha e^{-({k+\epsilon})-L_s^{K^{k+\epsilon}}} \eta_s\right)+(r-\mu_X) Z_s^z+r_c  \right) \left(-e^{-{k}}\right)\alpha \eta_s ds\notag\\
&\qquad\quad+e^{-\rho {\tau_\epsilon^*}}\tilde G_y\left(\alpha e^{-({k+\epsilon})-L_{\tau_\epsilon^*}^{K^{k+\epsilon}}}\eta_{\tau_\epsilon^*},Z^z_{\tau_\epsilon^*}\right)
\left(-e^{-{k}}\right)\alpha \eta_{\tau_\epsilon^*}\mathds 1_{\left\{\tau_\epsilon^*<\tau_f^{k+\epsilon}\right\}}\Bigg].
\end{align*}
On the other hand, we repeat the above argument for $\tau^*$ to have
\begin{align*}
u(k+\epsilon,z)-u(k,z)
&\geq \mathbb{E}\Bigg[\,\int_0^{\tau^*} e^{-\rho s} \Bigg(\tilde{U}\Big(\alpha e^{-(k+\epsilon)-L_s^{K^{k+\epsilon}}}\eta_s\Big)
      -\tilde{U}\Big(\alpha e^{-k-L_s^{K^k}}\eta_s\Big) \\
&\qquad\quad+((r-\mu_X)Z_s^z\alpha+r_c)\Big(e^{-(k+\epsilon)-L_s^{K^{k+\epsilon}}}
      -e^{-k-L_s^{K^k}}\Big)\eta_s\Bigg)ds \\
&\qquad\quad+e^{-\rho\tau^*}\Big(\tilde{G}\big(\alpha e^{-(k+\epsilon)-L_{\tau^*}^{K^{k+\epsilon}}}\eta_{\tau^*},Z_{\tau^*}^z\big)
      -\tilde{G}\big(\alpha e^{-k-L_{\tau^*}^{K^k}}\eta_{\tau^*},Z_{\tau^*}^z\big)\Big)\Bigg] \\
&\geq \mathbb{E}\Bigg[\,\int_0^{\tau^*\wedge\tau_f^{k}} e^{-\rho s} \Big(\tilde{U}'\big(\alpha e^{-k-L_s^{K^k}}\eta_s\big)
      +(r-\mu_X)Z_s^z+r_c\Big)\big(e^{-(k+\epsilon)}-e^{-k}\big)\alpha\eta_s\,ds \\
&\qquad\quad+e^{-\rho\tau^*}\tilde{G}_y\big(\alpha e^{-k-L_{\tau^*}^{K^k}}\eta_{\tau^*},Z_{\tau^*}^z\big)
      \big(e^{-(k+\epsilon)}-e^{-k}\big)\alpha\eta_{\tau^*}\mathds{1}_{\{\tau^*<\tau_f^{k}\}}\Bigg] \\
&\geq \epsilon\,\mathbb{E}\Bigg[\,\int_0^{\tau^*\wedge\tau_f^{k}} e^{-\rho s} \Big(\tilde{U}'\big(\alpha e^{-k-L_s^{K^k}}\eta_s\big)
      +(r-\mu_X)Z_s^z+r_c\Big)(-e^{-(k+\epsilon)})\alpha\eta_s\,ds \\
&\qquad\quad+e^{-\rho\tau^*}\tilde{G}_y\big(\alpha e^{-k-L_{\tau^*}^{K^k}}\eta_{\tau^*},Z_{\tau^*}^z\big)
      (-e^{-(k+\epsilon)})\alpha\eta_{\tau^*}\mathds{1}_{\{\tau^*<\tau_f^{k}\}}\Bigg].
\end{align*}
As a result, we arrive at
\begin{align*}
u_k(k,z)=\lim_{\epsilon\to 0} \frac{u(k+\epsilon,z)-u(k,z)}{\epsilon}=  &\Ex\bigg[\int_0^{\tau^* \wedge \tau_f^{k }} e^{-\rho s} \left(\tilde U'\left(\alpha e^{-{k}} \eta_s\right)+(r-\mu_X) Z^z_s +r_c \right) \left(-e^{-k}\right)\alpha \eta_s ds\\
&\quad+e^{-\rho {\tau^*}}\tilde G_y\left(\alpha e^{-{k}}\eta_{\tau^*},Z^z_{\tau^*}\right)
\left(-e^{-{k}}\right)\alpha \eta_{\tau^*}\mathds 1_{\left\{\tau^*<\tau_f^{k}\right\}}\bigg].
\end{align*}
Then, the desired result follows from the fact that $\tilde v_y(y,z)=-\frac{1}{y}u_k(k,z)$ for $(y,z)\in(0,\alpha]\times\R_+$. 
\hfill$\Box$\\

\noindent{\bf Proof of Lemma~\ref{lem:continuous}.}\quad
By using Lemma~\ref{lem:struct}-(ii), for any $z\in\mathbb{R}_+$, the dual function $y \mapsto \tilde{v}(y, z)$ is decreasing. Thus, to establish the continuity of $(y,z)\mapsto \hat{v}(y,z)$, it suffices to show that $\tilde{v}(y,z)$ is continuous in $z$ for any fixed $y\in (0, \alpha]$, and continuous in $y$ for any fixed $z\in \mathbb{R}_+$. First, we prove continuity of $z\mapsto\tilde{v}(y,z)$. To this purpose, fix $y\in (0,\alpha]$ and let $(z_n)_{n \geq 1}\subset\mathbb{R}_+$ satisfying  $z_n\to z$ as $n \to \infty$. Then, we have from \eqref{eq:tildev} that
\begin{align}\label{eq:tildev-v0}
&|\tilde v(y,z)-  \tilde v(y,z_n)|
  =\left|\sup_{\tau \in \mathbb T} J(y,z;\tau)-\sup_{\tau \in \mathbb T} J(y,z_n;\tau)\right|
\leq \sup_{\tau \in \mathbb T}|J(y,z;\tau)-J(y,z_n;\tau)|\\
&\quad\leq \sup_{\tau \in \mathbb T} \Bigg|\mathbb{E}\bigg[\int_0^{\tau } e^{-\rho s} \left((r-\mu_X)Y_s^{\alpha,y}-\ell\right) \left( Z_s^z- Z_s^{z_n}\right)ds\bigg]\Bigg| + \sup_{\tau \in \mathbb T} \Bigg|\mathbb{E}\bigg[e^{-\rho \tau}\left(\tilde G(Y_{\tau}^{\alpha,y},Z_{\tau}^z)-\tilde G(Y_{\tau}^{\alpha,y},Z_{\tau}^{z_n})\right)\bigg]\Bigg|\notag\\
&\quad\leq \int_0^\infty e^{-\rho s}\Ex\left[\left((\mu_X-r)Y_s^{\alpha,y}+\ell\right)\left| Z_s^z- Z_s^{z_n}\right|\right]ds
+ \sup_{\tau \in \mathbb T} \left|\mathbb{E}\left[e^{-\rho \tau}\left(\tilde G(Y_{\tau}^{\alpha,y},Z_{\tau}^z)-\tilde G(Y_{\tau}^{\alpha,y},Z_{\tau}^{z_n})\right)\right]\right|.\nonumber
\end{align}
Since $Y_s^{\alpha,y} \in(0,\alpha]$ for all $s\geq0$, Assumption $(\boldsymbol{A_{\rho}})$ and \eqref{eq:Zt} yield that
\begin{align}\label{eq:limit-Z-0}
\int_0^\infty e^{-\rho s}\mathbb{E}\left[\left((\mu_X-r)Y_s^{\alpha,y}+\ell\right)\left|Z_s^z-Z_s^{z_n}\right|\right]ds 
  &\leq\left((\mu_X-r)\alpha+\ell\right)|z-z_n| \int_0^\infty e^{-(\rho-\mu_Z)s}ds \nonumber \\
  &=\frac{\alpha|z-z_n|}{\rho-\mu_Z} \to 0, \quad n\to\infty.
\end{align}
For the 2nd term in \eqref{eq:tildev-v0}, it follows from  \eqref{eq:explicitformtildeG} that
\begin{align}\label{eq:argument33}
&\sup_{\tau \in \mathbb{T}}\left|\mathbb{E}\left[e^{-\rho \tau} \left(\tilde G(Y_\tau^{\alpha,y},Z_\tau^z)-\tilde G(Y_\tau^{\alpha,y},Z_\tau^{z_n})\right)\right]\right|
\leq \sup_{\tau \in \mathbb{T}}\mathbb{E}\left[e^{-\rho \tau}|g(Y_\tau^{\alpha,y})||Z_\tau^z-Z_\tau^{z_n}|\right] \nonumber \\
&\quad\leq |g(\alpha)| \sup_{\tau \in \mathbb{T}} \mathbb{E}\left[e^{-\rho \tau}|Z_\tau^z-Z_\tau^{z_n}|\right]= |g(\alpha)||z-z_n| \sup_{\tau \in \mathbb{T}}\mathbb{E}\left[e^{-(\rho-\mu_Z)\tau}\right] \nonumber \\
&\quad\leq |g(\alpha)||z-z_n| \to 0, \quad n\to \infty
\end{align}
with $g(y):= y - \beta^{-(\kappa - 1)}y^\kappa/\kappa$. From \eqref{eq:tildev-v0}, \eqref{eq:limit-Z-0} and \eqref{eq:argument33}, we can conclude that $\tilde{v}(y,z)$ is continuous in $z$ for any $y\in (0, \alpha]$. Recall that Lemma \ref{prop:tvy} implies that $\tilde{v}(y,z)$ is continuous in $y$; thus, we complete the proof of the lemma.
\hfill$\Box$\\

\noindent{\bf Proof of Lemma~\ref{lem:derivativev_G}.}\quad
Recall that the process $\eta=(\eta_t)_{t\geq0}$ is defined in \eqref{eq:Yrepresentation}. Let $Y^0_t:=y\eta_t$ for $t\geq0$. Then, we have
\begin{align*}
dY^{0}_t=\rho Y^{0}_tdt-\mu^\top (\sigma\sigma^\top)^{-1}\sigma Y^{0}_tdW_t,\quad Y^{0}_0=y\in(0,\alpha].
\end{align*}
For the $\mathbb{F}$-stopping time $\tau_f^k$ defined in Proposition \ref{prop:tvy}, integration by parts for $e^{-\rho \tau^*} Y_{\tau^*}^0 \tilde G_y(Y_{\tau^*}^0,Z_\tau^*) \mathds 1_{\{\tau_f^k>\tau^*\}}$ yields from \eqref{eq:tildevy} that
{\small\begin{align*}
& y\tilde v_y(y,z)- y\tilde G_y(y,z)\nonumber\\
&=\Ex\left[\int_0^{\tau^*(y,z)\wedge\tau_f^{-\ln\frac{y}{\alpha}}} e^{-\rho s}\left( \tilde U'(Y_s^0)+(r-\mu_X) Z_s^z+r_c+\tilde {\cal A} \tilde G_y(Y_s^0,Z_s^z) \right)Y_s^0 ds-e^{-\rho \tau_f^k} \tilde G_y(\alpha,Z_{\tau_f^k}) \alpha \mathds 1_{\{\tau_f^k\leq  \tau^*\}}\right].
\end{align*}} 
By virtue of Proposition 6.2 in \cite{bo2026extended}, we have that $\tilde G(y,z)$ satisfies the following linear PDE, for $(y,z)\in O_{\beta}:=(0,\beta)\times \R_+$, 
\begin{align}\label{eq:tildeGpde}
\begin{cases}
    \displaystyle \tilde {\cal L} \tilde G(y,z)-\mu_X  y z+\tilde U(y) = 0,~~\forall (y,z)\in O_{\beta},\\[0.6em]
    \displaystyle \tilde G_y(\beta,z)= 0,~~\forall z\in\R_+.
\end{cases}
\end{align}
This also implies from \eqref{eq:tildeGpde} that $\tilde G_y(y,z)$ satisfies the following PDE, on $(y,z)\in O_{\alpha}\times \R_+$,
 \begin{align}
\begin{cases}
\displaystyle \tilde {\cal A} u+(r-\mu_X)z+r_c +\tilde U'(y)=0,\\[0.5em]
\displaystyle u(\beta,z)=0,\quad\forall z\in\R_+.
\end{cases}
\end{align}
Then, we obtain $ y\tilde v_y(y,z)- y\tilde G_y(y,z)=-\alpha \Ex[e^{-\rho \tau_f^k} \tilde G_y(\alpha,Z_{\tau_f^k}) \mathds 1_{\{\tau_f^k\leq  \tau^*\}}]$. Since $\tilde G_y(y,z)\leq 0$ for all $(y,z)\in (0, \alpha]\times \R_+$, we derive $\tilde v_y(y,z)-\tilde G_y(y,z)\geq 0$ for all $(y,z)\in (0, \alpha]\times \R_+$. 
\hfill$\Box$\\

\noindent{\bf Proof of Lemma~\ref{lem:sectio}.}\quad
Define the sets by $\hat{\cal C}:=\left\{(y,z) \in (0,\alpha]\times\R_+;~ y>y_B(z)\right\}$ and $
\hat{\cal S}:=\left\{(y,z) \in (0,\alpha]\times\R_+;~ y\leq y_B(z)\right\}$. Recall the fact $\tilde{v}(y,z)\geq \tilde{G}(y,z)$ for all $(y,z)\in(0,\alpha]\times\R$. 
Then, using  Lemma~\ref{lem:derivativev_G}, for any $y\in(0,y_B(z)]$, we have $0\leq \tilde v(y,z)-\tilde G(y,z)\leq \tilde v(y_B(z),z)-\tilde G(y_B(z),z)=0$. Thus, we derive $\mathcal{\hat S}\subset \mathcal{S}$, which also results in $\mathcal C=\mathcal{S}^c \subset \mathcal{\hat S}^c=\mathcal {\hat C}$. On the other hand, for any $y>y_B(z)$, Lemma~\ref{lem:struct}-(iii) and \eqref{eq:stopingboundary} jointly imply that $\tilde v(y,z)> \tilde G(y,z)$. Consequently, we have $\mathcal {\hat C} \subset \mathcal C$, and hence $ \mathcal {\hat C} = \mathcal C$. This would also lead to $\mathcal S=\hat{\cal S}$.
\hfill$\Box$\\

\noindent{\bf Proof of Theorem~\ref{th:duality}.}\quad
For any $(\tau,\theta, c)\in\mathbb{T}\times{\mathbb{U}^{\rm r}}$, it follows from the budget constraint \eqref{eq:contrain} in Lemma~\ref{lem:XY} that, for any $(x,y,z)\in\R_+\times(0,\alpha]\times\R_+$, 
\begin{align}\label{eq:ineqv1}
&\mathbb{E}\left[ \int_0^\infty e^{-\rho t} (U(c_t)-\ell_t \mathds 1_{\{t< \tau\}} )dt- \alpha \int_0^{\tau} e^{-\rho t}dL_t^X- \beta \int_\tau^{\infty} e^{-\rho t}dL_t^X\right]+xy -xy\nonumber\\
&\leq\mathbb{E}\left[ \int_0^\infty e^{-\rho t} (U(c_t)-(\ell Z_t+\ell_c )\mathds 1_{\{t< \tau\}} )dt- \alpha \int_0^{\tau} e^{-\rho t}dL_t^X- \beta \int_\tau^{\infty} e^{-\rho t}dL_t^X\right]+xy\nonumber\\
&\quad-\mathbb{E}\bigg[\int_0^{\tau} e^{-\rho t}\left(c_t +(\mu_X-r)Z_t-r_c\right)Y_t^{\alpha} d t
+\int_0^{\tau} e^{-\rho t} X_t d L_t^{\alpha}-\int_0^{\tau} e^{-\rho t} Y_t^\alpha d L_t^X \notag\\
&\qquad+\int_{\tau}^\infty e^{-\rho t}\left(c_t +\mu_X Z_t \right)Y_t^{\tau,\beta} d t
+\int_\tau^{\infty} e^{-\rho t} X_t d L_t^{\tau,\beta}-\int_{\tau}^{\infty} e^{-\rho t} Y_t^{{\tau,\beta}} d L_t^X\bigg]\\
&=\Ex\left[\int_0^{\tau}e^{-\rho t}\{U(c_t)-\ell Z_t-\ell_c-(c_t+(\mu_X-r)Z_t-r_c)Y_t^{\alpha}\}dt\right]+\Ex\left[\int_{\tau}^{\infty}e^{-\rho t}\{U(c_t)-(c_t+\mu_X Z_t)Y_t^{\tau,\beta}\}dt\right]\nonumber\\
&\quad+\Ex\left[\int_0^{\tau}e^{-\rho t}(Y_t^{\alpha}-\alpha)dL_t^{X}+\int_{\tau}^{\infty}e^{-\rho t}(Y_t^{\tau,\alpha}-\beta)dL_t^{X}-\int_0^{\tau} e^{-\rho t} X_t d L_t^{\alpha}-\int_\tau^{\infty} e^{-\rho t} X_t d L_t^{\tau,\beta}\right]+xy.\notag
\end{align}
Recall the LF transform $\tilde{U}(\cdot)$ given by \eqref{eq:tildeU2} of the utility function $U(\cdot)$, i.e. $\tilde U(y):=\sup_{x> 0}\{U(x)-x y\}$ for $y\in(0,\beta]$. This yields that, for any $\tau\in\mathbb{T}$,
\begin{align*}
\begin{cases}
 \displaystyle  \Ex\left[\int_0^{\tau}e^{-\rho t}\{U(c_t)-\ell Z_t-\ell_c-(c_t+(\mu_X-r)Z_t-r_c)Y_t^{\alpha}\}dt\right]\\[0.4em]
 \displaystyle \qquad\qquad \leq  \Ex\left[\int_0^{\tau}e^{-\rho t}\{\tilde{U}(Y_t^{\alpha})+((r-\mu_X) Z_t+r_c) Y_t^{\alpha}-\ell Z_t-\ell_c\}dt\right],\\[0.4em]
 \displaystyle \Ex\left[\int_{\tau}^{\infty}e^{-\rho t}\{U(c_t)-(c_t+\mu_X Z_t)Y_t^{\tau,\beta}\}dt\right]\\[0.4em]
 \displaystyle\qquad\qquad \leq \Ex\left[\int_{\tau}^{\infty}e^{-\rho t}\{\tilde{U}(Y_t^{\tau,\beta})-\mu_X Z_tY_t^{\tau,\beta}\}dt\right].
\end{cases}
\end{align*}
On the other hand, note that $Y_t^{\alpha}\in(0,\alpha]$, $Y_t^{\tau,\beta}\in(0,\beta]$ and $X_t\geq0$, $\Px$-a.s. for all $t\in[0,T]$. Then, we have
\begin{align*}
    \Ex\left[\int_0^{\tau}e^{-\rho t}(Y_t^{\alpha}-\alpha)dL_t^{X}+\int_{\tau}^{\infty}e^{-\rho t}(Y_t^{\tau,\alpha}-\beta)dL_t^{X}-\int_0^{\tau} e^{-\rho t} X_t d L_t^{\alpha}-\int_\tau^{\infty} e^{-\rho t} X_t d L_t^{\tau,\beta}\right]\leq0.
\end{align*}
As a result, we deduce from \eqref{eq:ineqv1} that, for any $(\tau,\theta, c)\in\mathbb{T}\times{\mathbb{U}^{\rm r}}$ and $(x,y,z)\in\R_+\times(0,\alpha]\times\R_+$, 
\begin{align}\label{eq:ineqv2}
&\mathbb{E}\left[ \int_0^\infty e^{-\rho t} (U(c_t)-(\ell Z_t+\ell_c) \mathds 1_{\{t< \tau\}} )dt- \alpha \int_0^{\tau} e^{-\rho t}dL_t^X- \beta \int_\tau^{\infty} e^{-\rho t}dL_t^X\right]\\
&\leq\Ex\left[\int_0^{\tau}e^{-\rho t}\{\tilde{U}(Y_t^{\alpha})+((r-\mu_X)Z_t+r_c)Y_t^{\alpha}-(\ell Z_t+\ell_c) \}dt+\int_{\tau}^{\infty}e^{-\rho t}\{\tilde{U}(Y_t^{\tau,\beta})-\mu_X Z_tY_t^{\tau,\beta}\}dt\right]+xy.\nonumber
\end{align}
Recall the payoff function $\tilde{G}:(0, \beta] \times \mathbb{R}_+ \mapsto \mathbb{R}$ defined by \eqref{eq:tildeG}. Consequently, using the strong Markov property, it results in that
\begin{align}\label{eq:strongmarkovproperty}
 \Ex\left[\int_{\tau}^{\infty}e^{-\rho t}\left(\tilde{U}(Y_t^{\tau,\beta})-\mu_X Z_tY_t^{\tau,\beta}\right)dt\right]=e^{-\rho\tau}\tilde{G}(Y_{\tau}^{\beta},Z_{\tau}),
\end{align}
where the RDP $Y^{\beta}=(Y_t^{\beta})_{t\geq0}$ satisfies the reflected SDE \eqref{eq:Y_t} with $\alpha$ replaced by $\beta$. As a consequence, we have from \eqref{eq:ineqv2} and \eqref{eq:strongmarkovproperty} that 
\begin{align}\label{eq:tildev0}
  v(x,z)&=\sup\limits_{(\tau,{\theta},c)\in\mathbb{T}\times{\mathbb{U}^{\rm r}}}\mathbb{E}\left[ \int_0^\infty e^{-\rho t} (U(c_t)-(\ell Z_t +\ell_c) \mathds 1_{\{t< \tau\}} )dt- \alpha \int_0^{\tau} e^{-\rho t}dL_t^X- \beta \int_\tau^{\infty} e^{-\rho t}dL_t^X\right]\notag\\
 &\leq \sup_{\tau\in\mathbb T}\mathbb{E}\left[\int_0^\tau e^{-\rho t} \left(\tilde U(Y_t^\alpha)+((r-\mu_X) Z_t+r_c) Y_t^\alpha-(\ell Z_t+\ell_c) \right)dt+e^{-\rho \tau}\tilde G(Y_\tau^\alpha,Z_\tau)\right]+xy\nonumber\\
 &=\tilde{v}(y,z)+xy.
 \end{align}

We next prove that, for $(x,z)\in\R_+^2$ fixed, 
$J(x, z ; \tau^*,\theta^*, c^*)\geq \tilde v(y,z)+xy$ for any $y\in(0,\alpha]$ and the triplet $(\tau^*,\theta^*, c^*)$ being given in~\eqref{eq:optimalcontr}. Let $y^*(x,z)$ be the function satisfying $-\tilde v_y(y^*(x,z),z)=x$. Since $\tilde v_{yy}>0$, $\tilde v_y(\alpha,z)=0$ and $\lim_{y\to 0}\tilde v(y,z)=-\infty$, we have that $y^*(x,z)$ is well-defined. Denote by $X^*=(X^*)_{t\geq 0}$ the state process \eqref{state-X} under the triplet of control strategy $(\tau^*,\theta^*,c^*)$ specified by \eqref{eq:optimalcontr}. Introduce the state process $\tilde X^*=(\tilde X^*)_{t\geq 0}$ by
\begin{align}
\tilde X_t^*=\begin{cases}
\displaystyle -\tilde v_{y}(Y_t^{\alpha,*},Z_t), & \text{ for }{0\leq t<\tau^{*}},\\[0.4em]
\displaystyle  -\tilde G_{y}(Y_t^{\tau^*,\beta,*},Z_t), & \text{ for }{t\geq \tau^{*} },
\end{cases}\quad \tilde X_0^*=-\tilde v_y(y^*(x,z),z)=x,
\end{align}
where $Y^{\alpha,*}=(Y_t^{\alpha,*})_{0\leq t<\tau^*}$ is given by \eqref{eq:Y_t} with $Y_0^{\alpha,*}=y^*(y,z)$ and $(Y_{t}^{\tau^*,\beta,*})_{t\geq\tau^*}$ is driven by~\eqref{eq:Y_tbeta} with the initial value satisfying $Y_{\tau^*}^{\tau^* ,\beta,*}=Y_{\tau^* }^{\alpha,*}$. 

Let $\tau^{*,n}:=\tau^*\wedge n$ for $n\geq1$. Then, $(\tau^{*,n})_{n\geq1}$ is a sequence of $\Fx$-stopping times that is monotonically increasing and converges to $\tau^*$ as $n\to \infty$. We introduce the process $\tilde X^{*,n}=(\tilde X^{*,n})_{ t\geq 0}$ satisfying
\begin{align*}
\tilde X_t^{*,n}=\begin{cases}
\displaystyle -\tilde v_{y}(Y_t^{\alpha,*},Z_t), & \text{ for }{0\leq t<\tau^{*,n}},\\[0.4em]
\displaystyle  -\tilde G_{y}(Y_t^{{\tau^{*,n},\beta,*}},Z_t), & \text{ for }t\geq \tau^{*,n},
\end{cases}\quad \tilde X_0^{*,n}=-\tilde v_y(Y_0^{\alpha,*},Z_0)=-\tilde v_y(y^*(x,z),z)=x,
\end{align*}
where $(Y_{t}^{\tau^{*,n},\beta,*})_{t\geq \tau^{*,n}}$ is given by~\eqref{eq:Y_tbeta} with the initial value satisfying $Y_{\tau^*}^{\tau^{*,n},\beta,*}=Y_{\tau^{*,n}}^{\alpha,*}$. 
From the DCT and the continuous mapping theorem, it follows that, for any $t\geq0$, $\tilde X_t^{*,n} \overset{\rm a.s.} \to \tilde X_t^{*}$ and $Y_{t}^{\tau^{*,n},\beta,*}\overset{\rm a.s.} \to Y_{t}^{\tau^{*},\beta,*}$ as $n\to\infty$.

We now first consider the time horizon $t\in[0,\tau^{*,n})$ that $\tilde v_{yy}(Y_t^*,Z_t)$ is well-defined. For any $0\leq t<\tau^{*,n}$, by applying It\^o's formula, we have
\begin{align}\label{eq:tildeX*1}
&d\tilde X^{*,n}_t=-\big(\rho Y_t^{\alpha,*}\tilde v_{yy}(Y_t^{\alpha,*},Z_t)+\xi (Y_t^{\alpha,*})^2\tilde v_{yyy}(Y_t^{\alpha,*},Z_t) +\mu_Z (Z_t)\tilde v_{yz}(Y_t^{\alpha,*},Z_t)+\frac{1}{2}\sigma_Z^2(Z_t) \tilde v_{yzz}(Y_t^{\alpha,*},Z_t)\notag\\
&\qquad-\mu^\top \sigma ^{-1}\sigma_Z(Z_t) \gamma Y_t^{\alpha,*} \tilde v_{yyz}(Y_t^{\alpha,*},Z_t)\big) dt+\left(
\mu^\top \sigma ^{-1}  Y_t^{\alpha,*}\tilde v_{yy}(Y_t^{\alpha,*},Z_t)-\gamma\sigma_Z(Z_t)\tilde v_{yz}(Y_t^{\alpha,*},Z_t)\right)dW_t\nonumber\\
&\qquad+\tilde v_{yy}(Y_t^{\alpha,*},Z_t)dL_t^{\alpha},
\end{align}
where we recall that $\xi:=\mu^{\top}(\sigma \sigma^{\top})^{-1}\mu /2$. From Eq.~\eqref{eq:vihatv}, it follows that, on $\{t<\tau^{*,n}\}$,  
\begin{align}\label{eq:tleqtau}
&\rho Y_t^{\alpha,*}\tilde v_{y y}(Y_t^{\alpha,*}, Z_t) +\xi \tilde v_{y y y}(Y_t^{\alpha,*},Z_t) {(Y_t^{\alpha,*})^2}+\mu_Z Z_t \tilde v_{y z}(Y_t^{\alpha,*},Z_t)+\frac{\sigma_Z(Z_t)^2}{2}\tilde v_{y z z}(Y_t^{\alpha,*}, Z_t)\notag\\
&\qquad\qquad-\mu^\top \sigma ^{-1} \sigma_Z\gamma Y_t^{\alpha,*}\tilde v_{y y z}(Y_t^{\alpha,*}, Z_t) \notag\\
&\qquad= -2 \xi Y_t^{\alpha,*}\tilde v_{y y}(Y_t^{\alpha,*}, Z_t)+\mu^\top \sigma ^{-1} \sigma_Z\gamma
\tilde {v}_{y z}(Y_t^{\alpha,*}, Z_t)+\mu_X Z_t-\tilde U^{\prime}(Y_t^{\alpha,*})-r.
\end{align}
For any $ t\geq \tau^{*,n}$, by applying It\^o's formula again, we arrive at
\begin{align}\label{eq:tildeX*2}
&d\tilde X^{*,n}_t=-\bigg(\rho Y_t^{{\tau^{*,n},\beta,*}}\tilde G_{yy}(Y_t^{{\tau^{*,n},\beta,*}},Z_t)+\xi (Y_t^{{\tau^{*,n},\beta,*}})^2\tilde G_{yyy}(Y_t^{{\tau^{*,n},\beta,*}},Z_t) +\mu_Z (Z_t)\tilde G_{yz}(Y_t^{{\tau^{*,n},\beta,*}},Z_t)\notag\\
&\qquad+\frac{1}{2}\sigma_Z^2(Z_t) \tilde G_{yzz}(Y_t^{{\tau^{*,n},\beta,*}},Z_t)-\mu^\top \sigma ^{-1}\sigma_Z(Z_t) \gamma Y_t^{{\tau^{*,n},\beta,*}} \tilde G_{yyz}(Y_t^{{\tau^{*,n},\beta,*}},Z_t)\bigg) dt \\
&\qquad+\left(\mu^\top \sigma ^{-1}  Y_t^{{\tau^{*,n},\beta,*}}\tilde G_{yy}(Y_t^{{\tau^{*,n},\beta,*}},Z_t)-\gamma\sigma_Z(Z_t)\tilde G_{yz}(Y_t^{{\tau^{*,n},\beta,*}},Z_t)\right)dW_t+\tilde G_{yy}(Y_t^{{\tau^{*,n},\beta,*}},Z_t)dL_t^{Y^{*,n}}.\notag
\end{align}
Differentiating w.r.t. $y$ on both sides of~\eqref{eq:tildeGpde}, it holds that, on $\{t\geq \tau^{*,n}\}$,
\begin{align}\label{eq:tgeqtau}
&\rho Y_t^{{\tau^{*,n},\beta,*}}\tilde G_{y y}(Y_t^{{\tau^{*,n},\beta,*}}, Z_t)  +\xi \tilde G_{y y y}(Y_t^{{\tau^{*,n},\beta,*}},Z_t){(Y_t^{{\tau^{*,n},\beta,*}})^2}+\mu_Z Z_t \tilde G_{y z}(Y_t^{{\tau^{*,n},\beta,*}},Z_t)\notag\\
&\qquad\quad+\frac{\sigma_Z(Z_t)^2}{2}\tilde G_{y z z}(Y_t^{{\tau^{*,n},\beta,*}}, Z_t)-\mu^\top \sigma_Z(Z_t)\gamma Y_t^{{\tau^{*,n},\beta,*}}\tilde G_{y y z}(Y_t^{{\tau^{*,n},\beta,*}}, Z_t)\\
&\quad=  -2 \xi Y_t^{{\tau^{*,n},\beta,*}}\tilde G_{y y}(Y_t^{{\tau^{*,n},\beta,*}}, Z_t)+\mu^\top (\sigma^\top)^{-1}\sigma_Z Z_t \gamma
\tilde G_{y z}(Y_t^{{\tau^{*,n},\beta,*}}, Z_t)+\mu_X Z_t-\tilde U^{\prime}(Y_t^{{\tau^{*,n},\beta,*}}). \notag
\end{align}
Combining \eqref{eq:tildeX*1}-\eqref{eq:tgeqtau}, we obtain $\tilde U'(Y_t^{*,n})=-I_U(Y_t^{*,n})=:-c_t^{*,n}$, and
\begin{align*}
\theta_t^{*,n}:=\begin{cases}
\displaystyle \left(\sigma \sigma^{\top}\right)^{-1}\left(\mu Y_t^{{ \alpha,*}} \tilde{v}_{y y}\left(Y_t^{{\alpha,*}}, Z_t\right)+\sigma_Z Z_t  \sigma \gamma \tilde{v}_{y z}\left(Y_t^{{ \alpha,*}}, Z_t\right)-\sigma_Z Z_t \sigma \gamma\right),\text{ if } 0\leq t<\tau^{*,n},\\[0.6em]
\displaystyle \left(\sigma \sigma^{\top}\right)^{-1}\left(\mu Y_t^{{\tau^{*,n},\beta,*}} \tilde{G}_{y y}\left(Y_t^{{\tau^{*,n},\beta,*}}, Z_t\right)+\sigma_Z Z_t  \sigma \gamma \tilde{v}_{y z}\left(Y_t^{{\tau^{*,n},\beta,*}}, Z_t\right)-\sigma_Z Z_t  \sigma \gamma\right), \text{ if } t\geq \tau^{*,n}.
\end{cases}
\end{align*}
Then, we have
\begin{align}
d\tilde X^{*,n}_t &= (\theta^{*,n}_t)^\top\mu dt+r_t\mathds {1}_{\{t< \tau^{*,n}\}}dt+(\theta^{*,n}_t)^\top \sigma dW_t-c_t^{*,n} dt -\mu_Z Z_t dt\nonumber\\
&\quad-\sigma_Z Z_t  dW^\gamma_t+dL_t^{\tilde X^{*,n}},
\end{align}
where the process $L^{\tilde X^{*,n}}=( L_t^{\tilde X^{*,n}})_{t\geq0}$ is defined by, for any $t\geq0$,
\begin{align*}
 L_t^{\tilde X^{*,n}}&:=\int_0^t(\tilde v_{yy}(Y_s^{\alpha,*},Z_s)\mathds {1}_{\{s<\tau^{*,n}\}}+\tilde G_{yy}(Y_s^{n,\beta,*},Z_s)\mathds {1}_{\{s\geq \tau^{*,n}\}})dL_s^{Y^{*,n}}\\
 &=\int_0^t(\tilde v_{yy}(\alpha,Z_s)\mathds {1}_{\{s<\tau^{*,n}\}}+\tilde G_{yy}(\beta,Z_s)\mathds {1}_{\{s\geq \tau^{*,n}\}})dL_s^{Y^{*,n}}.   
\end{align*}  
Note that, for $(y,z)\in (0,\beta]\times\R_+$, we have
\begin{align*}
\tilde v_y(y,z)<0,~~ \tilde v_y(\alpha,z)=0,~~ \tilde v_{yy}(y,z)>0,~~ \tilde G_y(y,z)<0,~~ \tilde G_y(\beta,z)=0,~~\tilde G_{yy}(y,z)>0.
\end{align*}
This yields that $L^{\tilde{X}^{*,n}}=(L_t^{\tilde{X}^{*,n}})_{t \geq 0}$ is a continuous and non-decreasing process with $L_0^{\tilde{X}^{*,n}}=0$, which increases on the time set $\{t \geq 0; \tilde{X}_t^{*,n}=0\}$ only. This implies that $\tilde{X}^{*,n}=(\tilde{X}_t^{*,n})_{t\geq0}$ taking values on $[0, \infty)$ is a reflected process and $L^{\tilde{X}^{*,n}}$ is the local time process of $\tilde{X}^{*,n}$. Using the solution representation of ``the Skorokhod problem", we obtain $X_t^{*,n}=\tilde{X}_t^{*,n}$ for all $t \geq 0$, which implies that 
$X_t^{*,n}=0 \Longleftrightarrow Y_t^{\alpha,*}=\alpha \text{ if } t\in[0,\tau^{*,n})$ and $Y_t^{\tau^{*,n},\beta,*}=\beta \text{ if } t\in[\tau^{*,n},\infty)$.

It follows from \eqref{eq:u} that
\begin{align}\label{eq:cyt}
&v(x,z)\geq J(x,z;\tau^{*,n},\theta^{*,n},c^{*,n})=\Ex\Bigg[\int_0^{\tau^{*,n}}e^{-\rho t}(U(c_t^{*,n})-\ell_t)dt-\alpha\int_0^{\tau^{*,n}} e^{-\rho t} dL_t^{X^{*,n}}\notag\\
&\qquad\qquad+\int_{\tau^{*,n}}^{\infty}e^{-\rho t} U(c_t^{*,n}) dt-\beta\int_{\tau^{*,n}}^\infty e^{-\rho t} dL_t^{X^{*,n}}
\Bigg]\\
&=\Ex\left[\int_0^{\infty}e^{-\rho t}(\tilde U(Y_t^{*,n})+c^{*,n}_t Y_t^{*,n}-\ell_t\mathds 1_{\{t< \tau^{*,n}\}})dt-\alpha\int_0^{\tau^{*,n}} e^{-\rho t} dL_t^{X^{*,n}}-\beta\int_{\tau^{*,n}}^{\infty} e^{-\rho t} dL_t^{X^{*,n}}\right].\notag
\end{align}
Then, we have from It\^o's formula that 
\begin{align*}
& d\left(e^{-\rho t} X_t^{*,n} Y_t^{*,n}\right)=  -e^{-\rho t}\left(c_t^{*,n} +\mu_X Z_t  -r_t\mathds 1_{\{t< \tau^{*,n}\}}\right)Y_t^{*,n} d t-e^{-\rho t} X_t^{*,n} d L_t^{Y^{*,n}} \\
&\quad +e^{-\rho t} Y_t^{*,n} d L_t^{X^{*,n}}+e^{-\rho t}\left[\left(\mu^{\top} \sigma^{-1} X_t^{*,n} Y_t^{*,n}+\left(\theta_t^{*,n}\right)^{\top} \sigma X_t^{*,n} Y_t^{*,n}\right) d W_t-\sigma_Z Z_t  Y_t^{*,n} d W_t^\gamma\right].
\end{align*}
Plugging the above display into \eqref{eq:cyt},  it holds that
{\small\begin{align*}
& v(x,z)\geq \Ex\Bigg[\int_0^{\infty}e^{-\rho t}(\tilde U(Y_t^{*,n})-\mu_X Z_tY_t^{*,n}+r_t\mathds 1_{\{t< \tau^{*,n}\}}Y_t^{*,n}-\ell_t\mathds 1_{\{t< \tau^{*,n}\}})dt\notag\\
&\quad +\int_0^{\tau^{*,n}} e^{-\rho t} (Y_t^{*,n}-\alpha)dL_t^{X^{*,n}}+\int_{\tau^{*,n}}^{\infty} e^{-\rho t} (Y_t^{*,n}-\beta)dL_t^{X^{*,n}}-\int_0^{\infty} e^{-\rho t}X_t dL_t^{Y}\Bigg]+xy^*\notag\\
&=\Ex\left[\int_0^{\tau^{*,n}}e^{-\rho t}(\tilde U(Y_t^{*,n})+((r-\mu_X) Y_t^{*,n}-\ell) Z_t+r_c Y_t^{*,n}-\ell_c )dt+\int_{\tau^{*,n}}^\infty e^{-\rho t}(\tilde U(Y_t^{*,n})-\mu_X Y_t^{*,n}  Z_t )dt\right]+xy^*\notag\\
& =\Ex\left[\int_0^{\tau^{*,n}}e^{-\rho t}(\tilde U(Y_t^{*,n})+((r-\mu_X) Y_t^{*,n}-\ell) Z _t+r_c Y_t^{*,n}-\ell_c)dt+e^{-\rho \tau^{*,n}}\tilde G(Y_{\tau^{*,n}},Z_{\tau^{*,n}})\right]+xy^*.
\end{align*}} As a consequence, Fatou's lemma yields that
\begin{align*}
    \liminf_{n\to \infty } \Ex\left[\int_0^{\tau^{*,n}}e^{-\rho t} \tilde U(Y_t^{*,n}) dt+e^{-\rho \tau^{*,n}}\tilde G(Y_{\tau^{*,n}},Z_{\tau^{*,n}})
    \right]\geq  \Ex\left[\int_0^{\tau^{*}}e^{-\rho t} \tilde U(Y_t^{*}) dt+e^{-\rho \tau^{*}}\tilde G(Y_{\tau^{*}},Z_{\tau^{*}})\right].
\end{align*}
Under Assumption $\left(\boldsymbol{A}_{\boldsymbol{Z}}\right)$ and $\tau^{*,n}\leq\tau^*$, we deduce that
\begin{align*}
&   \left|\int_0^{\tau^{*,n}}e^{-\rho t}[((r-\mu_X) Y_t^{*,n}-\ell) Z _t+r_c Y_t^{*,n}-\ell_c ]dt\right| \leq
\int_0^{\tau^{*,n}}e^{-\rho t}|((r-\mu_X) Y_t^{*,n}-\ell) Z _t+r_c Y_t^{*,n}-\ell_c |dt\\
&\quad\leq \int_0^{\tau^{*}}e^{-\rho t}|((r-\mu_X) Y_t^{*,n}-\ell) Z _t+r_c Y_t^{*,n}-\ell_c |dt\leq \int_0^{\tau^{*}}e^{-\rho t}\left\{((\mu_X-r) \alpha+\ell) Z _t+|r_c\alpha-\ell_c |\right\}dt. 
\end{align*}
It is not difficult to check that $ \Ex\left[\int_0^{\tau^{*}}e^{-\rho t}((\mu_X-r) \alpha+\ell) Z _tdt\right]<+\infty$.  Using the DCT, we then arrive at 
\begin{align*}
&\lim_{n\to \infty } \Ex\left[\int_0^{\tau^{*,n}}e^{-\rho t}\left\{((r-\mu_X) Y_t^{*,n}-\ell) Z _t+r_c Y_t^{*,n}-\ell_c\right\}dt\right]\nonumber\\
&\qquad=\Ex\left[\int_0^{\tau^{* }}e^{-\rho t}\Big\{((r-\mu_X) Y_t^{* }-\ell) Z _t+r_c Y_t^*-\ell_c\Big\}dt
    \right].
\end{align*}
Consequently, we have 
\begin{align*}
v(x,z)&\geq \Ex\left[\int_0^{\tau^{* }}e^{-\rho t}(\tilde U(Y_t^{* })+((r-\mu_X) Y_t^{* }-\ell) Z _t+r_cY_t^{* }-\ell_c)dt+e^{-\rho \tau^{* }}\tilde G(Y_{\tau^{* }},Z_{\tau^{* }})\right]+xy^*\\
    &=\tilde v(y^*,z)+xy^*.
\end{align*}
Thus, we obtain $v(x, z)=\inf _{y \in(0, \alpha]}(\tilde{v}(y, z)+x y)$, and hence the proof of the dual theorem is complete. \hfill$\Box$\\

\noindent{\bf Proof of Corollary~\ref{coro:originalstrategy}.}\quad
Define $x^*(y,z) = v_x(\cdot,z)^{-1}(y)$ with $y \to v_x(\cdot,z)^{-1}(y)$ being the inverse function of $x \to v_x(x,z)$. Then, $x^* = x^*(y,z)$ satisfies the equation $v_x(x^*,z) = y$ for all $z \in \mathbb{R}_+$. We can obtain from a direct calculation that
\begin{align*}
\tilde{v}(y,z) &= v(x^*,z) - x^* y,\quad \tilde{v}_y(y,z) = -x^*,\quad \tilde{v}_z(y,z) = v_z(x^*,z), 
\end{align*}
and for $y> y_B(z)$, it holds that
\begin{align*}
\tilde{v}_{yy}(y,z) &= -\frac{1}{v_{xx}(x^*,z)},\quad \tilde{v}_{yz}(y,z) = \frac{v_{xz}(x^*,z)}{v_{xx}(x^*,z)}, \quad \tilde{v}_{zz}(y,z) = v_{zz}(x^*,z) - \frac{v_{xz}(x^*,z)^2}{v_{xx}(x^*,z)},
\end{align*}
while $0<y\leq y_B(z)$, it holds that
\begin{align*}
\tilde{G}_{yy}(y,z) &= -\frac{1}{G_{xx}(x^*,z)},\quad \tilde{G}_{yz}(y,z) = \frac{G_{xz}(x^*,z)}{G_{xx}(x^*,z)}, \quad \tilde{G}_{zz}(y,z) = G_{zz}(x^*,z) - \frac{G_{xz}(x^*,z)^2}{v_{xx}(x^*,z)}.
\end{align*}
By using the above dual relationship, it is not hard to check that the state process $X^*=(X_t^*)_{t\geq0}$ under the optimal strategy given by \eqref{eq:optimalcontr}-\eqref{eq:optimalcontr2} has the same dynamics as that under the feedback control triplet defined in this corollary.
\hfill$\Box$\\

\noindent{\bf Proof of Lemma~\ref{lem:ex1}.}\quad
Define $P(y):=\tilde v(y,0)-\tilde G(y,0)$ for $y\in (0,\alpha]$. Then, we have from the VI \eqref{eq:vihatv0} that, for $y\in(0,\alpha]$,
\begin{align*}
\xi y^2 P''(y)+\rho y P'(y)-\rho P(y)+r_c y-\ell_c  = 0,~~P(y_B) = 0,~~P'(y_B) = 0,~~ P'(\alpha) = M:=-\tilde G_y(\alpha,0).   
\end{align*}
We consider the following two sub-cases:
\begin{itemize}
\item {\bf Case 1}. In the case of $r_c\alpha\leq \ell_c$, we obviously have $P(y)=0$ for all $y\in(0,\alpha]$.
\end{itemize}

\begin{itemize}
\item {\bf Case 2}. In the case of $r_c\alpha>\ell_c$, we consider the following candidate solution to {Eq.~\eqref{eq:tildev}}:
\end{itemize}
\begin{align*}
P(y)=\left(C_1y+C_2y^{-\frac{\rho}{\xi}}-\frac{r_c}{\rho+\xi}y\ln y-\frac{\ell_c}{\rho}\right)\mathds 1_{\{y>y_B\}},\quad \forall y\in(0,\alpha],     
\end{align*}
with $0\leq y_B\leq\frac{\ell_c}{r_c}$. 
Then, the constant $C_1$, $C_2$ and $y_B$ satisfy
\begin{align}	\label{eq:z0}
\begin{cases} 
\displaystyle C_1y_B+C_2(y_B)^{-\frac{\rho}{\xi}}-\frac{r_c}{\rho+\xi} y_B\ln y_B-\frac{\ell_c}{\rho}= 0; \\[0.8em]
\displaystyle C_1-C_2\frac{\rho}{\xi}(y_B)^{-\frac{\rho}{\xi}-1}-\frac{r_c}{\rho+\xi}-\frac{r_c}{\rho+\xi}\ln y_B = 0; \\[0.8em]
\displaystyle C_1-C_2\frac{\rho}{\xi} \alpha^{-\frac{\rho}{\xi}-1}-\frac{r_c}{\rho+\xi}-\frac{r_c}{\rho+\xi}\ln\alpha= M.
\end{cases}
\end{align}
From the 1st and the 2nd equality in \eqref{eq:z0}, we deduce that 
\begin{align*}
\begin{cases} 
\displaystyle C_2(1+\frac{\rho}{\xi})(y_B)^{-\frac{\rho}{\xi}}+\frac{r_c}{\rho+\xi}y_B-\frac{\ell_c}{\rho}= 0\Longrightarrow C_2=\frac{\frac{\ell_c}{\rho}-\frac{r_c}{\rho+\xi}y_B}{1+\frac{\rho}{\xi}}(y_B)^{\frac{\rho}{\xi}};\\[0.6em]
\displaystyle C_1=C_2\frac{\rho}{\xi} (y_B)^{-\frac{\rho}{\xi}-1}+\frac{r_c}{\rho+\xi}+\frac{r_c}{\rho+\xi}\ln y_B = \frac{\frac{\rho}{\xi} \frac{\ell_c}{\rho}(y_B)^{-1}+\frac{r_c}{\rho+\xi}}{1+\frac{\rho}{\xi}}+\frac{r_c}{\rho+\xi}\ln y_B. 
\end{cases}
\end{align*}
Combining the 3rd equality in \eqref{eq:z0}, we derive that 
\begin{align}\label{eq:z04}
\frac{1}{\rho+\xi}\left(\frac{\ell_c/\alpha}{y_B/\alpha}-\frac{r_c\rho}{\rho+\xi}\right)
\left[1- \left(\frac{y_B}{\alpha}\right)^{\frac{\rho}{\xi}+1}\right]+\frac{r_c}{\rho+\xi} \ln\left(\frac{y_B}{\alpha}\right)- M=0.
\end{align}

Next, we prove that Eq.~\eqref{eq:z04} has a unique root $y_B\in(0,\ell_c/r_c]$. To do it, denote by 
\begin{align*}
F(x):=\frac{1}{\rho+\xi}\left(\frac{\ell_c/\alpha}{x}-\frac{r_c\rho}{\rho+\xi}\right)
\left[1- x^{\frac{\rho}{\xi}+1}\right]+\frac{r_c}{\rho+\xi} \ln x- M,\quad \forall x\in\left(0,\frac{\ell_c}{r_c\alpha}\right).   
\end{align*}
Then, we can check that 
\begin{align*}	
& F^\prime(x)=\left(\frac{1}{x^2}+\frac{\rho x^{\rho/\xi-1}}{\xi}\right)\frac{ r_c x-\ell_c/\alpha }{\rho+\xi}
<0,~~\forall x\in\left(0,\frac{\ell_c}{r_c\alpha}\right);\quad F(0+)=+\infty,\\[2mm]
& F(\ell_c/(r_c\alpha))=\frac{r_c}{\rho+\xi}G(\ell_c/(r_c\alpha))-M\;\;\mbox{with}\;\; G(x):=\frac{1- x^{\frac{\rho}{\xi}+1}}{1+\frac{\rho}{\xi}}+\ln x,\\[2mm]
& G^\prime(x)=- x^{\frac{\rho}{\xi}}+x^{-1}>0,~~\forall x\in(0,1),~~ 
G(1)=0;\quad G(\ell_c/(r_c\alpha))<0,~~F(\ell_c/(r_c\alpha))<0.	
\end{align*}
Hence, there exists a unique root $x^*\in(0,\ell_c/(r_c\alpha))$ of $F(x)=0$ on $x\in\left(0,\frac{\ell_c}{r_c\alpha}\right)$, and there exists a unique $y_B=\alpha X^*\in(0,\ell_c/r_c)$ such that \eqref{eq:z04} holds. Since $P(y)$ satisfies $\tilde {\mathcal{L}}P'(y) + r= 0$ for $y\in(y_B,\alpha)$ with $P'(y_B)= 0$ and $P'(\alpha)=M>0$. {Then, the strong comparison principle yields that $P'(y)>0$ for  $y\in(y_B,\alpha]$.} Combining the fact that $P(y_B) =0$, it follows that $P(y)>0$, for $y\in(y_B,\alpha]$. As a consequence $\tilde v(y,0)=\tilde G(y,0)+P(y)$, for $y\in(0,\alpha]$. 
\hfill$\Box$\\

\noindent{\bf Proof of Corollary~\ref{coro:ELS}.}\quad
Let $f(y):=\Ex[\int_0^{\tau^*}e^{-\rho t}\{U(I_U(Y_t^{\alpha,*}))-\ell_c\}dt]$ for $y\in(0,\alpha]$. It follows from It\^o's formula that, for $y\in(0,\alpha]$,
\begin{align*}
\xi y^2 f''(y)+\rho y f'(y)-\rho f(y)+\frac{1}{p}y^{\frac{p}{p-1}}-\ell_c=0, ~~f(y_B)=0,~~f'(\alpha)=0.
\end{align*}
Solving the above ODE, which yields that \begin{align*}
f(y)  &=\left(
\frac{\rho}{\xi}K-\frac{(p-1){\alpha^{\frac{1}{p-1}}}}{\rho(1-p)-\xi p}\right) y+ \alpha ^{{\rho \over \xi}+1}K  y^{-\frac{\rho}{\xi}}+\frac{(p-1)^2}{p[\rho(1-p)-\xi p]}y^{\frac{p}{p-1}}-\frac{\ell_c}{\rho}.  
\end{align*}
Define the function by $g(y):=\Ex[\int_{\tau^*}^\infty e^{-\rho t} U(I_U(Y_t^{\beta,*}))dt]=\Ex[e^{-\rho {\tau^*}} \hat g(Y_{\tau^*}^{\alpha,*})]$ with $\hat g(y):=\Ex[\int_0^{\infty}
e^{-\rho t}  U(I_U(Y_t^{\beta,*})) dt]$ for $y\in(0,\alpha]$. Then, similar to \eqref{eq:explicitformtildeG}, we have
\begin{align*}
\hat g(y)= \frac{(p-1)}{  \xi p + \rho(p-1) } \beta^{\frac{1}{p-1}} y - \frac{(p-1)^2}{p\left[p\xi + \rho(p-1) \right]} y^{\frac{p}{p-1}},\quad \forall y\in(0,\alpha].
\end{align*}        
Furthermore, we have from It\^o's rule that
\begin{align*}
\xi y^2 g''(y)+\rho y g'(y)-\rho g(y)=0, ~g(y_B)=\hat g(y_B),~g'(\alpha)=0,\quad \forall y\in(0,\alpha],
\end{align*}
which leads to that
\begin{align*}
g(y) = \frac{\hat{g}(y_B)}{\rho y_B + \xi \alpha^{\frac{\rho}{\xi}+1} y_B^{-\frac{\rho}{\xi}}} \left(\rho y + \xi \alpha^{\frac{\rho}{\xi}+1} y^{-\frac{\rho}{\xi}}\right),\quad \forall y\in(0,\alpha].
\end{align*}
Consequently, from the definition of problem~\eqref{eq:u} and the duality $y^*(x,0)=v_x(x,0)$, it follows that
\begin{align*}
&\Ex\left[ \alpha \int_0^{\tau^*} e^{-\rho t}dL_t^X+\beta \int_{\tau^*}^{\infty} e^{-\rho t}dL_t^X\right]=\Ex\left[ \int_0^{\tau^*} e^{-\rho t} (U(c_t^*)-\ell_c )dt+\int_{\tau^*}^\infty e^{-\rho t} U(c_t^*)dt\right]\nonumber\\
&\qquad- v(x,0)=f(v_x(x,0))+g(v_x(x,0))-v(x,0).
\end{align*}
Thus, we complete the proof of the corollary.
\hfill$\Box$\\

\noindent{\bf Proof of Lemma~\ref{lem:Pz}.}\quad
Let $P(y,z):=(\tilde v-\tilde G)(y,z)$ for $(y,z)\in (0,\alpha]\times \R_+$. Then, the free-boundary problem \eqref{eq:tildevGyBz} can be rewritten as: 
\begin{align*}
\begin{cases}
\displaystyle \tilde {\cal L} P+(ry -\ell) z+r_cy-\ell_c=0,~~\text{on}~O_{\alpha},\\[0.4em]
\displaystyle P_y(\alpha,z)=-\tilde G_y(\alpha,z),~\forall z\in\R_+,\\[0.4em]
\displaystyle P(y_B(z),z)=0,~P_y(y_B(z),z)=0,~\forall z\in\R_+.
\end{cases}
\end{align*}
As a consequence, for $(y,z)\in (y_B(z),\alpha]\times\R_+$, we deduce that $\tilde {\cal T} P_z+(ry -\ell) =0$ with the boundary condition $P_z(y_B(z),z)=0$. Here, the operator $\tilde {\cal T}$ acted on $\phi\in C^{2}(O_{\alpha})$ is defined by 
\begin{align*}
\tilde {\cal T}\phi &:= {\mu^{\top}(\sigma \sigma^{\top})^{-1} \mu \over2} y^2\partial_{yy}\phi+\sigma^2_Z z|\gamma|^2\partial_{z}\phi
 +{\sigma^2_Z z^2|\gamma|^2\over 2}\partial_{zz}\phi-\mu^{\top}(\sigma\sigma^{\top})^{-1}\sigma\sigma_Z z\gamma  y\partial_{yz}\phi
 +\rho y\partial_y\phi\nonumber\\
 &\quad+\mu_Z (\phi+z\partial_z\phi)-\rho \phi.    
\end{align*}
On the other hand, using VI~\eqref{eq:vihatv0}, it is not difficult to verify that, if $(y,z)\in \{(y,z)\in(0,\alpha]\times \R_+;~(ry -\ell) z+r_cy-\ell_c\leq  0\}$, then $P(y,z)\equiv 0$, i.e., $\tilde v (y,z)=\tilde G(y,z)$. In other words, if $\mathcal C$ is non-empty,  then  $y_B(z)\leq \frac{\ell z+\ell_c}{rz+r_c}$ holds for all $z\in \R_+$. If $r_c=\ell_c=0$, it follows from the maximum principle in \cite{gilbarg1977elliptic} that $P_z(y,z)\leq 0$ on $\{(y,z)\in(0,\alpha]\times\R_+;~y_B(z)<y\leq \frac{\ell}{r}\}$; while $P_z\geq 0$ on $\{(y,z)\in(0,\alpha]\times\R_+;~y\in (\frac{\ell}{r},\alpha]\}$. Consequently, it is not difficult to have $y_B(z_1)\leq y_B(z_2)$ for $z_1\leq z_2$. 
\hfill$\Box$\\

\end{document}